\documentclass[a4paper,12pt,oneside,reqno]{amsart}

\usepackage[T1]{fontenc}
\usepackage[utf8]{inputenc}
\usepackage{lmodern}
\usepackage[english]{babel}

\usepackage[%
	left=2.5cm,       
	right=2.5cm,      
	top=3.5cm,        
	bottom=3.5cm,     
	heightrounded,    
	bindingoffset=0mm 
]{geometry}

\usepackage{amsmath}
\usepackage{amssymb}
\usepackage{mathtools}
\usepackage{mathrsfs}

\usepackage{hyperref} 
\usepackage[noabbrev,capitalize]{cleveref}

\numberwithin{equation}{section}

\theoremstyle{plain}
	\newtheorem{theorem}{Theorem}[section]
	\newtheorem{lemma}[theorem]{Lemma}
	\newtheorem{proposition}[theorem]{Proposition}

\theoremstyle{definition}
	\newtheorem{definition}[theorem]{Definition}
	
	\newtheorem{remark}[theorem]{Remark}
	
\usepackage{calrsfs}	

\newcommand{\N}{\mathbb{N}}
\newcommand{\R}{\mathbb{R}}
\newcommand{\G}{\mathbb{G}}

\newcommand{\leb}{\mathcal{L}}
\newcommand{\de}{\partial}
\newcommand{\eps}{\varepsilon}

\newcommand{\weakto}{\rightharpoonup}
\newcommand{\prob}{\mathcal{P}}
\newcommand{\h}{\mathsf{h}}
\newcommand{\F}{\mathsf{F}}
\newcommand{\ch}{\mathsf{Ch}}
\newcommand{\dd}{\mathsf{d}}
\newcommand{\dcc}{\mathsf{d}_{\mathsf{cc}}}
\newcommand{\deps}{\mathsf{d}_\eps}
\newcommand{\ent}{\mathsf{Ent}}
\newcommand{\W}{\mathsf{W}}
\newcommand{\Bcc}{\mathsf{B}_\mathbb{G}}
\newcommand{\D}{\mathrm{D}}
\newcommand{\m}{\mathfrak{m}}

\renewcommand{\phi}{\varphi}
\renewcommand{\rho}{\varrho}
\renewcommand{\theta}{\vartheta}

\DeclareMathOperator{\supp}{supp}

\DeclareMathOperator{\dom}{Dom}

\DeclareMathOperator{\diag}{diag}
\DeclareMathOperator{\vol}{vol}
\DeclareMathOperator{\ric}{Ric}

\DeclareMathOperator{\diverg}{div}
\DeclareMathOperator{\tang}{Tan}
\DeclareMathOperator{\Lip}{Lip}
					    
\DeclarePairedDelimiter{\scalar}{<}{>}                                     
\DeclarePairedDelimiter{\set}{\{}{\}}
\DeclarePairedDelimiter{\abs}{|}{|}

\mathchardef\ordinarycolon\mathcode`\:
\mathcode`\:=\string"8000
\begingroup \catcode`\:=\active
  \gdef:{\mathrel{\mathop\ordinarycolon}}
\endgroup

\usepackage{bookmark}

\usepackage[initials]{amsrefs}

\usepackage{tikz}
\usepackage{graphicx}
\usepackage[font=small]{caption}

\usepackage{enumerate}

\usepackage{url}

\allowdisplaybreaks

\begin{document}

\title{Heat and entropy flows in Carnot groups}

\author[L. Ambrosio]{Luigi Ambrosio}
\address{Scuola Normale Superiore, Piazza Cavalieri 7, 56126 Pisa, Italy}
\email{luigi.ambrosio@sns.it}

\author[G. Stefani]{Giorgio Stefani}
\address{Scuola Normale Superiore, Piazza Cavalieri 7, 56126 Pisa, Italy}
\email{giorgio.stefani@sns.it}

\date{\today}



\begin{abstract}
We prove the correspondence between the solutions of the sub-elliptic heat equation in a Carnot group~$\G$ and the gradient flows of the relative entropy functional in the Wasserstein space of probability measures on~$\G$. Our result completely answers a question left open in a previous paper by N.~Juillet, where the same correspondence was proved for $\G=\mathbb{H}^n$, the $n$-dimensional Heisenberg group.   
\end{abstract}

\maketitle

\section{Introduction}

Since the pioneering works~\cites{JKO98,O01} and the monograph~\cite{AGS08}, in the last twenty years there has been an increasing interest in the study of the relation between evolution equations and gradient flows of energy functionals in a large variety of different frameworks, see~\cites{AGS14,AS07,CM14,E10,EM14,FSS10,G10,GKO13,GO12,KL09,M11,O09,OS09,S07,V09}. 

The prominent case in the literature is represented by the connection between the heat equation and the relative entropy functional. It is well-known that the heat equation 
\begin{equation}\label{intro_eq:heat_eq}\tag{\textbf{HE}}
\begin{cases}
\de_t u_t=\Delta u_t & \text{in}\ (0,+\infty)\times\R^n,\\
u_0=\bar{u}\in L^2(\R^n) & \text{on}\ \set*{0}\times\R^n,
\end{cases}
\end{equation} 
can be seen as the gradient flow in $L^2(\R^n)$ of the Dirichlet energy $\mathsf{D}(u)=\int_{\R^n}|\nabla u|^2\ dx$ accordingly to the general approach introduced in~\cite{B73}. If the initial datum $\bar u\in L^2(\R^n)$ is such that $\mu_0=\bar{u}\,\leb^n\in\prob_2(\R^n)$, where 
\begin{equation*}
\prob_2(\R^n)=\set*{\mu\in\prob(\R^n) : \int_{\R^n} |x|^2\ d\mu(x)<+\infty},
\end{equation*}
then the solution $(u_t)_{t\ge0}$ of~\eqref{intro_eq:heat_eq} induces a curve $(\mu_t)_{t\ge0}\subset\prob_2(\R^n)$, $\mu_t=u_t\,\leb^n$. If we endow the set $\prob_2(\R^n)$ with its usual Wasserstein distance~$\W$, then the curve $(\mu_t)_{t\ge0}$ is locally absolutely continuous with locally integrable squared $\W$-derivative in $[0,+\infty)$. On the one hand, since $(u_t)_{t\ge0}$ satisfies~\eqref{intro_eq:heat_eq}, the curve $(\mu_t)_{t\ge0}$ naturally solves in the weak sense the continuity equation
\begin{equation}\label{intro_eq:CE}\tag{\textbf{CE}}
\begin{cases}
\de_t\mu_t+\diverg(v_t\mu_t)=0 & \text{in}\ (0,+\infty)\times\R^n,\\
\mu_0=\bar{u}\,\leb^n & \text{on}\ \set*{0}\times\R^n,
\end{cases}
\end{equation} 
where the velocity vector field $(v_t)_{t\ge0}$ is given by $v_t=-{\nabla u_t}/{u_t}$. On the other hand, the \emph{relative entropy}
\begin{equation*}
\ent(\mu)=\int_{\R^n} u\log u\ dx, 
\qquad
\text{for}\ \mu=u\,\leb^n\in\prob_2(\R^n),
\end{equation*}
computed along a curve $(\mu_t)_{t\ge0}\subset\prob_2(\R^n)$ satisfying~\eqref{intro_eq:CE} for a given velocity field $(v_t)_{t\ge0}$ is such that 
\begin{equation*}
\frac{d}{dt}\ent(\mu_t)=-\int_{\R^n} (\log u_t+1)\,\diverg(v_t u_t)\ dx=\int_{\R^n}\scalar*{v_t,\frac{\nabla u_t}{u_t}}\ d\mu_t.
\end{equation*} 
In analogy with the Hilbertian case, but using Otto's calculus~\cite{O01} in the interpretation of the right hand side, one says that $(\mu_t)_{t\ge0}$ is a gradient flow of the entropy in $(\prob_2(\R^n),\W)$ if and only if the curve $t\mapsto\ent(\mu_t)$ has maximal dissipation rate. This happens if and only if $v_t=-{\nabla u_t}/{u_t}$, i.e.\ when $(u_t)_{t\ge0}$ satisfies~\eqref{intro_eq:heat_eq}.  

Although not fully rigorous, the argument presented above contains all the key tools needed to establish the correspondence between the heat flow and the entropy flow in a general metric measure space $(X,\dd,\m)$. Both the heat equation~\eqref{intro_eq:heat_eq} and the continuity equation~\eqref{intro_eq:CE} have been adequately understood in this general context. For~\eqref{intro_eq:heat_eq}, one relaxes the Dirichlet energy to the so-called \emph{Cheeger energy}
\begin{equation*}
\ch(u)=\inf\set*{\liminf_n\int_X |\D u_n|^2\ d\m : u_n\to u\ \text{in}\ L^2(X,\dd,\m),\ u_n\in\Lip(X)},
\end{equation*}
where the \emph{local Lipschitz constant} $|\D u|(x)=\limsup\limits_{y\to x}\frac{|u(y)-u(x)|}{\dd(x,y)}$ of $u\in\Lip(X)$ plays the same role of the absolute value of the gradient in~$\R^n$. It can be shown that the naturally associated Sobolev space $W^{1,2}(X,\dd,\m)$ is a Banach space (not Hilbertian in general) and that the functional~$\ch$ is convex, so that~\eqref{intro_eq:heat_eq} can still be interpreted as its gradient flow in the Hilbert space $L^2(X,\dd,\m)$, see~\cite{AGS14}. For~\eqref{intro_eq:CE}, one introduces an appropriate space $\textsf{S}^2(X)$ of test functions in $W^{1,2}(X,\dd,\m)$ and says that $(\mu_t)_{t>0}$ satisfies the continuity equation with respect to a family of maps $(L_t)_{t>0}\colon\mathsf{S}^2(X)\to\R$ if $t\mapsto\int_X f\ d\mu_t$ is absolutely continuous for every $f\in\mathsf{S}^2(X)$ with $\frac{d}{dt}\int_X f\ d\mu_t=L_t(f)$ for a.e.\ $t>0$, see~\cite{GH15}.

The notion of gradient flow of the entropy functional 
\begin{equation*}
\ent_\m(\mu)=\int_X \rho\log\rho\ d\m, 
\qquad
\text{for}\ \mu=\rho\m\in\prob_2(X),
\end{equation*}
in the Wasserstein space $(\prob_2(X),\W)$ can be rigorously defined by requiring the validity of the following 
sharp \textit{energy dissipation inequality}
\begin{equation*}
\ent_\m(\mu_t)+\frac{1}{2}\int_s^t|\dot{\mu}_r|^2\ dr+\frac{1}{2}\int_s^t|\D^-\ent_\m|^2(\mu_r)\ dr\le \ent_\m(\mu_s)
\end{equation*}
for all $0\le s\le t$, where $t\mapsto|\dot{\mu}_t|$ is the $\W$-derivative of the curve $(\mu_t)_{t>0}\subset\prob_2(X)$ and 
\begin{equation*}
|\D^-\ent_\m|(\mu)=\limsup\limits_{\nu\to\mu}\max\set*{\frac{\ent_\m(\mu)-\ent_\m(\nu)}{\W(\mu,\nu)},0}
\end{equation*}
is the so-called \emph{descending slope} of the entropy. Note that this definition is consistent with the standard one in Hilbert spaces, 
since $u'(t)=-\nabla E(u(t))$ is equivalent to 
\begin{equation*}
\frac{1}{2}|u'|^2(t)+\frac{1}{2}|\nabla E(u(t))|^2\le-\frac{d}{dt} E(u(t))
\end{equation*}
by combining the chain rule with Cauchy--Schwarz and Young's inequalities.

As pointed out in~\cites{G10,AGS14}, this abstract approach provides a complete equivalence between the two gradient flows if the entropy is \emph{$K$-convex along geodesics} in $(\prob_2(X),\W)$ for some $K\in\R$, that is, if
\begin{equation}\label{intro_eq:K-convexity}\tag{\textbf{K}}
\ent_\m(\mu_t)\le(1-t)\ent_\m(\mu_0)+t\,\ent_\m(\mu_1)-\frac{K}{2}t(1-t)\W(\mu_0,\mu_1)^2,
\qquad
t\in[0,1],
\end{equation}
holds for a class of constant speed geodesics $(\mu_t)_{t\in[0,1]}\subset\prob_2(X)$ sufficiently large to join any pair of points
in~$\prob_2(X)$. The $K$-convexity (also known as \emph{displacement convexity}) of the entropy heavily depends on the structure of $(X,\dd,\m)$ and encodes a precise information about the ambient space: if $X$ is a Riemannian manifold, then~\eqref{intro_eq:K-convexity} is valid if and only if the Ricci curvature satisfies $\mathrm{Ric}\ge K$, see~\cite{vRS05}. For this reason, if property~\eqref{intro_eq:K-convexity} holds, then $(X,\dd,\m)$ is called a space with \emph{generalized Ricci curvature} bounded from below, 
or simply a $CD(K,\infty)$ space. 

According to this general framework, the correspondence between heat flow and entropy flow has been proved on Riemannian manifolds with Ricci curvature bounded from below, see~\cite{E10}, and on compact \emph{Alexandrov spaces}, see~\cites{GKO13,GO12,O09}. Alexandrov spaces are considered as metric measure spaces with \emph{generalized sectional curvature} bounded from below (a condition stronger than~\eqref{intro_eq:K-convexity}, see~\cite{P11}).

If $(X,\dd,\m)$ is not a $CD(K,\infty)$ space, then the picture is less clear. As stated in~\cite{AGS14}*{Theorem~8.5}, the correspondence between heat flow and entropy flow still holds if the descending slope $|\D^-\ent_\m|$ of the entropy is an upper gradient of~$\ent_\m$ and satisfies a precise lower semicontinuity property, basically equivalent to the equality between $|\D^-\ent_\m|$ and the so-called \emph{Fisher information}. These assumptions are weaker than~\eqref{intro_eq:K-convexity} but not easy to check for a given non-$CD(K,\infty)$ space.

In~\cites{J14,J09}, it was proved that the \emph{Heisenberg group}~$\mathbb{H}^n$ is a non-$CD(K,\infty)$ space in which nevertheless the correspondence between heat flow and gradient flow holds. The Heisenberg group is the simplest non-commutative \emph{Carnot group}. Carnot groups are one of the most studied examples of Carnot--Carathéodory spaces, see~\cites{BLU07,LeD16,M02} and the references therein for an account on this subject. The proof of the correspondence of the two flows in~$\mathbb{H}^n$ presented in~\cite{J14} essentially splits into two parts. The first part shows that a solution of the sub-elliptic heat equation $\de_t u_t+\Delta_{\mathbb{H}^n} u_t=0$ corresponds to a gradient flow of the entropy in the Wasserstein space $(\prob_2(\mathbb{H}^n),\W_{\mathbb{H}^n})$ induced by the Carnot-Carathéodory distance~$\dcc$. The direct computations needed are justified by some precise estimates on the \emph{sub-elliptic heat kernel} in~$\mathbb{H}^n$ given in~\cites{L06,L07}. The second part proves that a gradient flow of the entropy in $(\prob_2(\mathbb{H}^n),\W_{\mathbb{H}^n})$ induces a sub-elliptic heat diffusion in~$\mathbb{H}^n$. The argument is based on a clever regularization of the gradient flow  $(\mu_t)_{t\ge0}$ based on the particular structure of the Lie algebra of~$\mathbb{H}^n$.

An open question arisen in~\cite{J14}*{Remark~5.3} was to extend the correspondence of the two flows to any Carnot group. The aim of the present work is to give a positive answer to this problem. To prove that a solution of the sub-elliptic heat equation corresponds to a gradient flow of the entropy, we essentially follow the same strategy of~\cite{J14}. Since the results of~\cites{L06,L07} are not known for a general Carnot group, we instead rely on the weaker estimates given in~\cite{VSC92} valid in any nilpotent Lie group. To show that a gradient flow of the entropy induces a sub-elliptic heat diffusion, we regularize the gradient flow $(\mu_t)_{t\ge0}$ both in time and space via convolution with smooth kernels. This regularization does not depend on the structure of the Lie algebra of the group, but nevertheless allows us to preserve the key quantities involved, such as the \emph{continuity equation} and the \emph{Fisher information}. In the presentation of the proofs, we also take advantage of a few results taken from the general setting of metric measures spaces developed in~\cite{AGS14} and in the references therein.

The paper is organized as follows. In \cref{sec:preliminaries} we collect the standard definitions and well-known facts that are used throughout the work. The precise statement of our main result is given in \cref{th:main} at the end of this part. In \cref{sec:CE_and_ent_slope} we extend the technical results presented in~\cite{J14}*{Sections~3 and~4} to any Carnot group with minor modifications and we prove that Carnot groups are non-$CD(K,\infty)$ spaces (see \cref{prop:carnot_not_CD}), generalizing the analogous result obtained in~\cite{J09}. Finally, in \cref{sec:proof_of_main_result}, we prove the correspondence of the two flows.  

\section{Preliminaries}
\label{sec:preliminaries}

\subsection{AC curves, entropy and gradient flows}

Let $(X,\dd)$ be a metric space, let $I\subset\R$ be a closed interval and let $p\in[1,+\infty]$. We say that a curve $\gamma\colon I\to X$ belongs to $AC^p(I;(X,d))$ if 
\begin{equation}\label{eq:def_AC_curve}
\dd(\gamma_s,\gamma_t)\le\int_s^t g(r)\ dr 
\qquad
s,t\in I,\ s<t, 
\end{equation} 
for some $g\in L^p(I)$. The space $AC^p_{\rm loc}(I;(X,d))$ is defined analogously. The case $p=1$ corresponds to \emph{absolutely continuous curves} and is simply denoted by $AC(I;(X,d))$. It turns out that, if $\gamma\in AC^p(I;(X,d))$, there is a minimal function $g\in L^p(I)$ satisfying~\eqref{eq:def_AC_curve}, called \emph{metric derivative} of the curve~$\gamma$, which is given by
\begin{equation*}
|\dot{\gamma}_t|:=\lim_{s\to t}\frac{d(\gamma_s,\gamma_t)}{|s-t|}
\qquad
\text{for a.e.}\ t\in I.
\end{equation*} 
See~\cite{AGS08}*{Theorem~1.1.2} for the simple proof. We call $(X,d)$ a \emph{geodesic} metric space if for every $x,y\in X$ there exists a curve $\gamma\colon[0,1]\to X$ such that $\gamma(0)=x$, $\gamma(1)=y$ and
\begin{equation*}
\dd(\gamma_s,\gamma_t)=|s-t|\,\dd(\gamma_0,\gamma_1) 
\qquad
\forall s,t\in[0,1].
\end{equation*}

Let $\R^*=\R\cup\set{-\infty,+\infty}$ and let $f\colon X\to\R^*$ be a function. We define the \emph{effective domain} of~$f$ as
\begin{equation*}
\dom(f):=\set*{x\in X : f(x)\in\R}.
\end{equation*} 
Given $x\in\dom(f)$, we define the \emph{local Lipschitz constant} of~$f$ at~$x$ by
\begin{equation*}
|\D f|(x):=\limsup_{y\to x}\frac{|f(y)-f(x)|}{\dd(x,y)}.
\end{equation*}
The \emph{descending slope} and the \emph{ascending slope} of~$f$ at~$x$ are respectively given by
\begin{equation*}
|\D^-f|(x):=\limsup_{y\to x}\frac{[f(y)-f(x)]^-}{\dd(x,y)},
\qquad
|\D^+f|(x):=\limsup_{y\to x}\frac{[f(y)-f(x)]^+}{\dd(x,y)}.
\end{equation*}
Here $a^+$ and $a^-$ denote the positive and negative part of $a\in\R$ respectively. When $x\in\dom(f)$ is an isolated point of $X$, we set $|\D f|(x)=|\D^-f|(x)=|\D^+f|(x)=0$. By convention, we set $|\D f|(x)=|\D^-f|(x)=|\D^+f|(x)=+\infty$ for all $x\in X\setminus\dom(f)$.

\begin{definition}[Gradient flow]\label{def:metric_GF}
Let $E\colon X\to\R\cup\set*{+\infty}$ be a function. We say that a curve $\gamma\in AC_{\rm loc}([0,+\infty);(X,\dd))$ is a \emph{(metric) gradient flow} of~$E$ starting from $\gamma_0\in\dom(E)$ if the \emph{energy dissipation inequality (EDI)}
\begin{equation}\label{eq:def_EDI}
E(\gamma_t)+\frac{1}{2}\int_s^t|\dot{\gamma}_r|^2\ dr+\frac{1}{2}\int_s^t|\D^-E|^2(\gamma_r)\ dr\le E(\gamma_s)
\end{equation}
holds for all $s,t\ge0$ with $s\le t$. 
\end{definition}

Note that, if $(\gamma_t)_{t\ge0}$ is a gradient flow of~$E$, then $\gamma_t\in\dom(E)$ for all $t\ge0$ and $\gamma\in AC^2_{\rm loc}([0,+\infty);(X,\dd))$ with $t\mapsto|\D^-E|(\gamma_t)\in L^2_{\rm loc}([0,+\infty))$. Moreover, the function $t\mapsto E(\gamma_t)$ is non-increasing on~$[0,+\infty)$ and thus a.e.\ differentiable and locally integrable.

\begin{remark}\label{rem:GF_AC}
As observed in~\cite{AGS14}*{Section~2.5}, if the function $t\mapsto E(\gamma_t)$ is locally absolutely continuous on~$(0,+\infty)$, then~\eqref{eq:def_EDI} holds as an equality by the chain rule and Young's inequality. In this case, \eqref{eq:def_EDI} is also equivalent to 
\begin{equation*}
\frac{d}{dt}E(\gamma_t)=-|\dot{\gamma}_t|^2=-|\D^-E|^2(\gamma_t)
\qquad
\text{for a.e.}\ t>0.
\end{equation*}
\end{remark}

\subsection{Wasserstein space}
\label{subsec:wasserstein_space}

We now briefly recall some properties of the Wasserstein space needed for our purposes. For a more detailed introduction to this topic, we refer the interested reader to~\cite{AG13}*{Section~3}. 

Let $(X,\dd)$ be a Polish space, i.e.\ a complete and separable metric space. We denote by $\prob(X)$ the set of probability Borel measures on~$X$. The \emph{Wasserstein distance} $\W$ between $\mu,\nu\in\prob(X)$ is given by
\begin{equation}\label{eq:def_W_2}
\W(\mu,\nu)^2=\inf\set*{\int_{X\times X} \dd(x,y)^2\ d\pi : \pi\in\Gamma(\mu,\nu)},
\end{equation}
where
\begin{equation}\label{eq:def_coupling}
\Gamma(\mu,\nu):=\set*{\pi\in\prob(X\times X) : (p_1)_\#\pi=\mu,\ (p_2)_\#\pi=\nu}.
\end{equation}
Here $p_i\colon X\times X\to X$, $i=1,2$, are the the canonical projections on the components. As usual, if $\mu\in\prob(X)$ and $T\colon X\to Y$ is a $\mu$-measurable map with values in the topological space~$Y$, the \emph{push-forward measure} $T_\#(\mu)\in\prob(Y)$ is defined by $T_\#(\mu)(B):=\mu(T^{-1}(B))$ for every Borel set $B\subset Y$. The set $\Gamma(\mu,\nu)$ introduced in~\eqref{eq:def_coupling} is call the set of \emph{admissible plans} or \emph{couplings} for the pair~$(\mu,\nu)$. For any Polish space $(X,\dd)$, there exist optimal couplings where the infimum in~\eqref{eq:def_W_2} is achieved.

The function~$\W$ is a distance on the so-called \emph{Wasserstein space} $(\prob_2(X),\W)$, where
\begin{equation*}
\prob_2(X):=\set*{\mu\in\prob(X) : \int_X \dd(x,x_0)^2\ d\mu(x)<+\infty\ 
\text{for some, and thus any,}\ x_0\in X}.
\end{equation*}
The space $(\prob_2(X),\W)$ is Polish. If $(X,\dd)$ is geodesic, then $(\prob_2(X),\W)$ is geodesic as well. Moreover, $\mu_n\xrightarrow{\W}\mu$ if and only if $\mu_n\weakto\mu$ and
\begin{equation*}
\int_X \dd(x,x_0)^2\ d\mu_n(x)\to\int_X \dd(x,x_0)^2\ d\mu(x)\ \qquad\text{for some}\ x_0\in X.
\end{equation*}
As usual, we write $\mu_n\weakto\mu$ if $\int_X\phi\ d\mu_n\to\int_X\phi\ d\mu$ for all $\phi\in C_b(X)$.

\subsection{Relative entropy}

Let $(X,\dd,\m)$ be a metric measure space, where $(X,\dd)$ is a Polish metric space and $\m$ is a non-negative, Borel and $\sigma$-finite measure. We assume that the space $(X,\dd,\m)$ satisfies the following structural assumption: there exist a point $x_0\in X$ and two constants $c_1,c_2>0$ such that 
\begin{equation}\label{eq:assumption_on_measure}
\m\left(\set*{x\in X : \dd(x,x_0)<r}\right)\le c_1 e^{c_2 r^2}.
\end{equation}
The \emph{relative entropy} $\ent_\m\colon\prob_2(X)\to(-\infty,+\infty]$ is defined as
\begin{equation}\label{eq:def_entropy}
\ent_\m(\mu):=
\begin{cases}
\displaystyle\int_X \rho\log\rho\ d\m & \text{if}\ \mu=\rho\m\in\prob_2(X) ,\\[3mm]
+\infty & \text{otherwise}.
\end{cases}
\end{equation}
According to our definition, $\mu\in\dom(\ent_\m)$ implies that $\mu\in\prob_2(X)$ and that the effective domain $\dom(\ent_\m)$ is convex. As pointed out in~\cite{AGS14}*{Section~7.1}, the structural assumption~\eqref{eq:assumption_on_measure} guarantees that in fact $\ent(\mu)>-\infty$ for all $\mu\in\prob_2(X)$.

When $\m\in\prob(X)$, the entropy functional~$\ent_\m$ naturally extends to~$\prob(X)$, is lower semicontinuous with respect to the weak convergence in~$\prob(X)$ and positive by Jensen's inequality. In addition, if $F\colon X\to Y$ is a Borel map, then
\begin{equation}\label{eq:ent_push-forward_formula}
\ent_{F_\#\m}(F_\#\mu)\le\ent_\m(\mu)
\qquad
\text{for all}\ \mu\in\prob(X),
\end{equation}  
with equality if~$F$ is injective, see~\cite{AGS08}*{Lemma~9.4.5}.

When $\m(X)=+\infty$, if we set $\mathfrak{n}:=e^{-c\,\dd(\cdot,x_0)^2}\m$ for some $x_0\in X$, where $c>0$ is chosen so that $\mathfrak{n}(X)<+\infty$ (note that the existence of such $c>0$ is guaranteed by~\eqref{eq:assumption_on_measure}), then we obtain the useful formula
\begin{equation}\label{eq:ent_useful_formula}
\ent_\m(\mu)=\ent_{\mathfrak{n}}(\mu)-c\int_X\dd(x,x_0)^2\ d\mu
\qquad
\text{for all}\ \mu\in\prob_2(X).
\end{equation}
This shows that $\ent_\m$ is lower semicontinuous in $(\prob_2(X),\W)$.

\subsection{Carnot groups}
\label{subsec:carnot_groups}

Let $\G$ be a Carnot group, i.e.\ a connected, simply connected and nilpotent Lie group whose Lie algebra $\mathfrak{g}$ of left-invariant vector fields has dimension $n$ and admits a stratification of step $\kappa$, 
\begin{equation*}
\mathfrak{g}=V_1\oplus V_2\oplus\cdots\oplus V_\kappa
\end{equation*}
with
\begin{equation*}
V_i=[V_1,V_{i-1}]\quad \text{for } i=1,\dots,\kappa, \qquad [V_1,V_\kappa]=\set{0}. 
\end{equation*}
We set $m_i=\dim(V_i)$ and $h_i=m_1+\dots+m_i$ for $i=1,\dots,\kappa$, with $h_0=0$ and $h_\kappa=n$. We fix an adapted basis of $\mathfrak{g}$, i.e.\ a basis $X_1,\dots,X_n$ such that
\begin{equation*}
X_{h_{i-1}+1},\dots,X_{h_i}\ \text{is a basis of}\ V_i,\qquad i=1,\dots,\kappa.
\end{equation*}
Using exponential coordinates, we can identify~$\G$ with $\R^n$ endowed with the group law determined by the Campbell--Hausdorff formula (in particular, the identity $e\in\G$ corresponds to $0\in\R^n$ and $x^{-1}=-x$ for $x\in\G$). It is not restrictive to assume that $X_i(0)=\mathrm{e}_i$ for any $i=1,\dots,n$; therefore, by left-invariance, for any $x\in\G$ we get
\begin{equation}\label{eq:X_i_dleft_transl}
X_i(x)=dl_x\mathrm{e}_i, \qquad i=1,\dots,n,
\end{equation}
where $l_x\colon\G\to\G$ is the left-translation by $x\in\G$, i.e.\ $l_x(y)=xy$ for any $y\in\G$. We endow $\mathfrak{g}$ with the left-invariant Riemannian metric $\scalar*{\cdot,\cdot}_\G$ that makes the basis $X_1,\dots,X_n$ orthonormal. For any $i=1,\dots,n$, we define the gradient with respect to the layer $V_i$ as
\begin{equation*}
\nabla_{V_i} f:=\sum_{j=h_{i-1}+1}^{h_i} (X_j f) \, X_j\in V_i.
\end{equation*}
We let $H\G\subset T\G$ be the \emph{horizontal tangent bundle} of the group~$\G$, i.e.\ the left-invariant sub-bundle of the tangent bundle~$T\G$ such that $H_e\G=\set*{X(0) : X\in V_1}$. We use the distinguished notation $\nabla_\G:=\nabla_{V_1}$ for the \emph{horizontal gradient}.

For any $i=1,\dots,n$, we define the degree $d(i)\in\set*{1,\dots,\kappa}$ of the basis vector field $X_i$ as $d(i)=j$ if and only if $X_i\in V_j$. With this notion, the one-parameter family of group dilations $(\delta_\lambda)_{\lambda\ge0}\colon\G\to\G$ is given by
\begin{equation}\label{eq:def_dilation}
\delta_\lambda(x)=\delta_\lambda(x_1,\dots,x_n):=(\lambda x_1,\dots,\lambda^{d(i)} x_i,\dots,\lambda^\kappa x_n) \qquad\text{for all}\ x\in\G.
\end{equation}
The Haar measure of the group $\G$ coincides with the $n$-dimensional Lebesgue measure~$\leb^n$ and has the homogeneity property $\leb^n(\delta_\lambda(E))=\lambda^Q\leb^n(E)$, where the integer $Q=\sum_{i=1}^\kappa i\dim(V_i)$ is the \emph{homogeneous dimension} of the group. 

We endow the group $\G$ with the canonical Carnot--Carathéodory structure induced by~$H\G$. We say that a Lipschitz curve $\gamma\colon[0,1]\to\G$ is a \emph{horizontal curve} if $\dot{\gamma}(t)\in H_{\gamma(t)}\G$ for a.e.\ $t\in[0,1]$. The \emph{Carnot--Carathéodory distance} between $x,y\in\G$ is then defined as
\begin{equation*}
\dcc(x,y)=\inf\set*{\int_0^1\|\dot{\gamma}(t)\|_\G\ dt : \gamma\ \text{is horizontal},\ \gamma(0)=x,\ \gamma(1)=y}.
\end{equation*}
By Chow--Rashevskii's Theorem, the function $\dcc$ is in fact a distance, which is also left-invariant and homogeneous with respect to the dilations defined in~\eqref{eq:def_dilation}, precisely $\dcc(zx,zy)=\dcc(x,y)$ and $\dcc(\delta_\lambda(x),\delta_\lambda(y))=\lambda\dcc(x,y)$ for all $x,y,z\in\G$ and $\lambda\ge0$. The resulting metric space $(\G,\dcc)$ is a Polish geodesic space. We let $\Bcc(x,r)$ be the $\dcc$-ball centred at $x\in\G$ of radius $r>0$. Note that $\leb^n(\Bcc(x,r))=c_n r^Q$, where $c_n=\leb^n(\Bcc(0,1))$. In particular, the metric measure space $(\G,\dcc,\leb^n)$ satisfies the structural assumption~\eqref{eq:assumption_on_measure}. 

Let us write $x=(\tilde{x}_1,\dots,\tilde{x}_\kappa)$, where $\tilde{x}_i:=(x_{h_{i-1}+1},\dots,x_{h_i})$ for $i=1,\dots,\kappa$. As proved in~\cite{FSSC03}*{Theorem~5.1}, there exist suitable constants $c_1=1$, $c_2,\dots,c_k\in(0,1)$ depending only on the group structure of~$\G$ such that
\begin{equation}\label{eq:def_box_norm}
\dd_\infty(x,0):=\max\set*{c_i\abs*{\tilde{x}_i}_{\R^{m_i}}^{1/i} : i=1,\dots,\kappa}, \qquad x\in\G,
\end{equation}
induces a left-invariant and homogeneous distance $\dd_\infty(x,y):=\dd_\infty(y^{-1}x,0)$, $x,y\in\G$, which is equivalent to~$\dcc$.

Let $1\le p<+\infty$ and let $\Omega\subset\R^n$ be an open set. The \emph{horizontal Sobolev space}
\begin{equation}\label{eq:def_horiz_sobolev_space}
W^{1,p}_\G(\Omega):=\set*{u\in L^p(\Omega) :  X_i u\in L^p(\Omega),\ i=1,\dots,m_1}
\end{equation}
endowed with the norm
\begin{equation*}
\|u\|_{W^{1,p}_\G(\Omega)}:=\|u\|_{L^p(\Omega)}+\sum_{i=1}^{m_1}\|X_i u\|_{L^p(\Omega)}
\end{equation*}
is a reflexive Banach space, see~\cite{FSSC96}*{Proposition~1.1.2}. By~\cite{FSSC96}*{Theorem~1.2.3}, the set $C^\infty(\Omega)\cap W^{1,p}_\G(\Omega)$ is dense in $W^{1,p}_\G(\Omega)$. By a standard cut-off argument, we get that $C^\infty_c(\R^n)$ is dense in $W^{1,p}_\G(\R^n)$.

\subsection{Riemannian approximation}
\label{subsec:riemannian_approx}

The metric space $(\G,\dcc)$ can be seen as the limit in the \emph{pointed Gromov--Hausdorff sense} as $\eps\to0$ of a family of Riemannian manifolds $\set*{(\G_\eps,\deps)}_{\eps>0}$ defined as follows, see~\cite{CDPT07}*{Theorem~2.12}. For any $\eps>0$, we define the Riemannian approximation $(\G_\eps,\deps)$ of the Carnot group $(\G,\dcc)$ as the manifold $\R^n$ endowed with the Riemannian metric $g_\eps(\cdot,\cdot)\equiv\scalar*{\cdot,\cdot}_\eps$ that makes orthonormal the vector fields $\eps^{d(i)-1}X_i$, $i=1,\dots,n$, i.e.\ such that
\begin{equation*}
\scalar*{X_i,X_j}_\eps=\eps^{2-d(i)-d(j)}\delta_{ij}, \qquad i,j=1,\dots,n.
\end{equation*}
We let $\deps$ be the Riemannian distance induced by the metric~$g_\eps$. Note that $\deps$ is left-invariant and satisfies $\deps\le\dcc$ for all $\eps>0$. For any $\eps>0$, the $\eps$-Riemannian gradient is defined as  
\begin{equation*}
\nabla_\eps f=\sum_{i=1}^n\eps^{2(d(i)-1)}(X_i f) \, X_i=\sum_{i=1}^\kappa\eps^{2(i-1)}\nabla_{V_i}f.
\end{equation*}
By~\eqref{eq:X_i_dleft_transl}, we get that
\begin{equation*}
g_\eps(x)=(dl_x)^T D_\eps\,(dl_x), \qquad x\in\G,
\end{equation*} 
where $D_\eps$ is the diagonal block matrix given by
\begin{equation*}
D_\eps=\diag(\mathbf{1}_{m_1},\eps^{-2}\mathbf{1}_{m_2},\dots,\eps^{-2(i-1)}\mathbf{1}_{m_i},\dots,\eps^{-2(\kappa-1)}\mathbf{1}_{m_\kappa}).
\end{equation*}
As a consequence, the Riemannian volume element is given by 
\begin{equation*}
\vol_\eps=\sqrt{\det g_\eps}\,dx_1\wedge\dots\wedge dx_n=\eps^{n-Q}\leb^n.
\end{equation*}
We remark that, for each $\eps>0$, the $n$-dimensional Riemannian manifold $(\G_\eps,\deps)$ has Ricci curvature bounded from below. More precisely, there exists a constant $K>0$, depending only on the Carnot group $\G$, such that
\begin{equation}\label{eq:ricci_bounded_below}
\ric_\eps\ge-K\eps^{-2}
\qquad
\text{for all}\ \eps>0.
\end{equation}
By scaling invariance, the proof of inequality~\eqref{eq:ricci_bounded_below} can be reduced to the case $\eps=1$, which in turn is a direct consequence of~\cite{M76}*{Lemma~1.1}.

In the sequel, we will consider the metric measure space $(\G_\eps,\deps,\leb^n)$, i.e.\ the Riemannian manifold $(\G_\eps,\deps,\textrm{vol}_\eps)$ with a rescaled volume measure. Both these two spaces satisfy the structural assumption~\eqref{eq:assumption_on_measure}. Moreover, we have
\begin{equation}\label{eq:entropy_riem_vs_leb}
\ent_{\textrm{vol}_\eps}(\mu)=\ent(\mu)+\log(\eps^{Q-n})
\qquad
\text{for all}\ \eps>0.
\end{equation}
Here and in the following, $\ent$ denotes the entropy with respect to the reference measure~$\leb^n$.

\subsection{Sub-elliptic heat equation}

We let $\Delta_\G=\sum_{i=1}^{m_1} X_i^2$ be the so-called \emph{sub-Laplacian operator}. Since the horizontal vector fields $X_1,\dots,X_{h_1}$ satisfy H\"ormander's condition, by H\"ormander's theorem the \emph{sub-elliptic heat operator} $\de_t-\Delta_\G$ is \emph{hypoelliptic}, meaning that its fundamental solution $\h\colon(0,+\infty)\times\G\to(0,+\infty)$, $\h_t(x)=\h(t,x)$, the so-called \emph{heat kernel}, is smooth. In the following result, we collect some properties of the heat kernel that will be used in the sequel. We refer the reader to~\cite{VSC92}*{Chapter~IV} and to the references therein for the proof.

\begin{theorem}[Properties of the heat kernel]\label{th:property_heat_kernel}
The heat kernel $\h\colon(0,+\infty)\times\G\to(0,+\infty)$ satisfies the following properties:
\begin{enumerate}[(i)]
\item $\h_t(x^{-1})=\h_t(x)$ for any $(t,x)\in(0,+\infty)\times\G$;
\item $\h_{\lambda^2 t}(\delta_\lambda(x))=\lambda^{-Q}\h_t(x)$ for any $\lambda>0$ and $(t,x)\in(0,+\infty)\times\G$;
\item $\int_\G\h_t\ dx=1$ for any $t>0$; 
\item there exists $C>0$, depending only on $\G$, such that
\begin{equation}\label{eq:heat_estimate_above}
\h_t(x)\le Ct^{-Q/2}\exp\left(-\frac{\dcc(x,0)^2}{4t}\right)
\qquad
\forall (t,x)\in(0,+\infty)\times\G;
\end{equation}
\item for any $\eps>0$, there exists $C_\eps>0$ such that
\begin{equation}\label{eq:heat_estimate_below}
\h_t(x)\ge C_\eps t^{-Q/2}\exp\left(-\frac{\dcc(x,0)^2}{4(1-\eps)t}\right)
\qquad
\forall (t,x)\in(0,+\infty)\times\G;
\end{equation}
\item for every $j,l\in\N$ and $\eps>0$, there exists $C_\eps(j,l)>0$ such that
\begin{equation}\label{eq:heat_estimate_derivatives}
|(\de_t)^l X_{i_1}\cdots X_{i_j}\h_t(x)|\le C_\eps(j,l)t^{-\frac{Q+j+2l}{2}}\exp\left(-\frac{\dcc(x,0)^2}{4(1+\eps)t}\right)
\qquad
\forall (t,x)\in(0,+\infty)\times\G,
\end{equation}
where $X_{i_1}\cdots X_{i_j}\in V_1$.
\end{enumerate}
\end{theorem} 

Given $\rho\in L^1(\G)$, the function 
\begin{equation}\label{eq:sol_heat_diff}
\rho_t(x)=(\rho\star\h_t)(x)=\int_\G \h_t(y^{-1} x)\, \rho(y)\ dy,
\qquad
(t,x)\in(0,+\infty)\times\G,
\end{equation}
is smooth and is a solution of the heat diffusion problem
\begin{equation}\label{eq:heat_diffusion}
\begin{cases}
\de_t\rho_t=\Delta_\G\rho_t &\text{in}\ (0,+\infty)\times\G,\\
\rho_0=\rho, &\text{on}\ \{0\}\times\G.
\end{cases}
\end{equation}
The initial datum is assumed in the $L^1$-sense, i.e.\ $\lim\limits_{t\to0}\rho_t=\rho$ in $L^1(\G)$. As a consequence of the properties of the heat kernel, if $\rho\ge0$ then the solution $(\rho_t)_{t\ge0}$ in~\eqref{eq:sol_heat_diff} is everywhere positive and satisfies 
\begin{equation*}
\int_\G\rho_t(x)\ dx=\|\rho\|_{L^1(\mathbb{G})}
\qquad
\forall t>0.
\end{equation*}
In addition, if $\rho\leb^n\in\prob_2(\G)$ then $(\rho_t\leb^n)_{t\ge 0}\subset\prob_2(X)$. Indeed, by~\eqref{eq:heat_estimate_above}, we have
\begin{equation*}
C_t:=\int_\G\dcc(x,0)^2\,\h_t(x)\ dx<+\infty
\qquad
\forall t>0.
\end{equation*} 
Thus, by triangular inequality, we have
\begin{equation*}
(\dcc(\cdot,0)^2\star\h_t)(x)
=\int_\G\dcc(xy^{-1},0)^2\,\h_t(y)\ dy
\le2\,\dcc(x,0)^2+2C_t,
\end{equation*}
so that, for all $t>0$, we get
\begin{equation}\label{eq:heat_second_moment}
\int_\G\dcc(x,0)^2\,\rho_t(x)\ dx
=\int_\G(\dcc(\cdot,0)^2\star\h_t)(x)\,\rho(x)\ dx
\le2\int_\G\dcc(x,0)^2\,\rho(x)\ dx+2C_t.
\end{equation}

\subsection{Main result}

We are now ready to state the main result of the paper. The proof is given in \cref{sec:proof_of_main_result} and deals with the two parts of the statement separately. 

\begin{theorem}\label{th:main}
Let $(\G,\dcc,\leb^n)$ be a Carnot group and let $\rho_0\in L^1(\G)$ be such that $\mu_0=\rho_0\leb^n\in\dom(\ent)$. If $(\rho_t)_{t\ge0}$ solves the sub-elliptic heat equation $\de_t\rho_t=\Delta_\G\rho_t$ with initial datum $\rho_0$, then $\mu_t=\rho_t\leb^n$ is a gradient flow of $\ent$ in $(\prob_2(\G),\W_\G)$ starting from~$\mu_0$.

Conversely, if $(\mu_t)_{t\ge0}$ is a gradient flow of $\ent$ in $(\prob_2(\G),\W_\G)$ starting from~$\mu_0$, then  $\mu_t=\rho_t\leb^n$ for all $t\ge0$ and $(\rho_t)_{t\ge0}$ solves the sub-elliptic heat equation $\de_t\rho_t=\Delta_\G\rho_t$ with initial datum $\rho_0$.
\end{theorem}

\section{Continuity equation and slope of the entropy}
\label{sec:CE_and_ent_slope}

\subsection{The Wasserstein space on the approximating Riemannian manifold}

Let $(\prob_2(\G_\eps),\W_\eps)$ be the Wasserstein space introduced in \cref{subsec:wasserstein_space} relative to the metric measure space $(\G_\eps,\deps,\leb^n)$. As observed in \cref{subsec:riemannian_approx}, $(\G_\eps,\deps,\leb^n)$ is an $n$-dimensional Riemannian manifold (with rescaled volume measure) whose Ricci curvature is bounded from below. Here we collect some known results taken from~\cites{E10,V09} concerning the space $(\prob_2(\G_\eps),\W_\eps)$. In the original statements, the canonical reference measure is the Riemannian volume. Keeping in mind that $\textrm{vol}_\eps=\eps^{n-Q}\leb^n$ and the relation~\eqref{eq:entropy_riem_vs_leb}, in our statements each quantity is rescaled accordingly. All time-dependent vector fields appearing in the sequel are tacitly understood to be Borel measurable. 

Let $\mu\in\prob_2(\G_\eps)$ be given. We define the space 
\begin{equation*}
L^2_\eps(\mu)=\set*{\xi\in\mathcal{S}(T\G_\eps) : \int_{\G_\eps}\|\xi\|_\eps^2\ d\mu<+\infty}.
\end{equation*}
Here $\mathcal{S}(T\G_\eps)$ denotes the set of sections of the tangent bundle $T\G_\eps$. Moreover, we define the \emph{`tangent space'} of $(\prob_2(\G_\eps),\W_\eps)$ at~$\mu$ as
\begin{equation*}
\tang_\eps(\mu)=\overline{\set*{\nabla_\eps\phi : \phi\in C^\infty_c(\R^n)}}^{L^2_\eps(\mu)}.
\end{equation*}
The `tangent space' $\tang_\eps(\mu)$ was first introduced in~\cite{O01}. We refer the reader to~\cite{AGS08}*{Chapter~12} and to~\cite{V09}*{Chapters~13 and~15} for a detailed discussion on this space.

Let $\eps>0$ be fixed. Given $I\subset\R$ an open interval and a time-dependent vector field $v^\eps\colon I\times\G_\eps\to T\G_\eps$, $(t,x)\mapsto v^\eps_t(x)\in T_x\G_\eps$, we say that a curve $(\mu_t)_{t\in I}\subset\prob_2(\G_\eps)$ satisfies the \emph{continuity equation} 
\begin{equation}\label{eq:CE_eps}
\de_t\mu_t+\diverg(v^\eps_t\mu_t)=0 \qquad\text{in}\ I\times\G_\eps
\end{equation}
\emph{in the sense of distributions} if 
\begin{equation*}
\int_I\int_{\G_\eps}\|v^\eps_t(x)\|_\eps\,d\mu_t(x)\,dt<+\infty
\end{equation*}
and
\begin{equation*}
\int_I\int_{\G_\eps}\de_t\phi(t,x)+\scalar*{v^\eps_t(x),\nabla_\eps\phi(t,x)}_\eps\,d\mu_t(x)\,dt=0
\qquad
\forall \phi\in C^\infty_c(I\times\R^n).
\end{equation*} 
We can thus state the following result, see~\cite{E10}*{Proposition~2.5} for the proof. Here and in the sequel, the metric derivative in the Wasserstein space $(\prob_2(\G_\eps),\W_\eps)$ of a curve $(\mu_t)_{t\in I}\subset\prob_2(\G_\eps)$ is denoted by~$|\dot{\mu}_t|_\eps$.

\begin{proposition}[Continuity equation in $(\prob_2(\G_\eps),\W_\eps)$]\label{prop:CE_eps}
Let $\eps>0$ be fixed and let $I\subset\R$ be an open interval. If $(\mu_t)_t\in AC^2_{\rm loc}(I;\prob_2(\G_\eps))$, then there exists a time-dependent vector field $v^\eps\colon I\times\G_\eps\to T\G_\eps$ with $t\mapsto\|v^\eps_t\|_{L^2_\eps(\mu_t)}\in L^2_{\rm loc}(I)$ such that
\begin{equation}\label{eq:v_eps}
v^\eps_t\in\tang_\eps(\mu_t) 
\qquad
\text{for a.e.}\ t\in I
\end{equation}
and the continuity equation~\eqref{eq:CE_eps} holds in the sense of distributions. The vector field~$v^\eps_t$ is uniquely determined in $L^2_\eps(\mu_t)$ by~\eqref{eq:CE_eps} and~\eqref{eq:v_eps} for a.e.\ $t\in I$ and we have
\begin{equation*}
\|v^\eps_t\|_{L^2_\eps(\mu_t)}=|\dot{\mu}_t|_\eps
\qquad
\text{for a.e.}\ t\in I.
\end{equation*}

Conversely, if $(\mu_t)_{t\in I}\subset\prob_2(\G_\eps)$ is a curve satisfying~\eqref{eq:CE_eps} for some $(v^\eps_t)_{t\in I}$ such that $t\mapsto\|v^\eps_t\|_{L^2_\eps(\mu_t)}\in L^2_{\rm loc}(I)$, then $(\mu_t)_t\in AC^2_{\rm loc}(I;(\prob_2(\G_\eps),\W_\eps))$ with 
\begin{equation*}
|\dot{\mu}_t|_\eps\le\|v^\eps_t\|_{L^2_\eps(\mu_t)}
\qquad
\text{for a.e.}\ t\in I.
\end{equation*}
\end{proposition}

\noindent
We can interpret the time-dependent vector field $(v^\eps_t)_{t\in I}$ given by \cref{prop:CE_eps} as the `tangent vector' of the curve $(\mu_t)_{t\in I}$ in $(\prob_2(\G_\eps),\W_\eps)$. As remarked in~\cite{E10}*{Section~2}, for a.e.\ $t\in I$ the vector field $v^\eps_t$ has minimal $L^2_\eps(\mu_t)$-norm among all time-dependent vector fields satisfying~\eqref{eq:CE_eps}. Moreover, this minimality is equivalent to~\eqref{eq:v_eps}.

In the following result and in the sequel, $|\D_\eps^-\ent|(\mu)$ denotes the descending slope of the entropy~$\ent$ at the point $\mu\in\prob_2(\G_\eps)$ in the Wasserstein space $(\prob_2(\G_\eps),\W_\eps)$.

\begin{proposition}\label{prop:slope_ent_eps}
Let $\eps>0$ be fixed and let $\mu=\rho\leb^n\in\prob_2(\G_\eps)$. The following statements are equivalent: 
\begin{enumerate}[(i)]
\item $|\D_\eps^-\ent|(\mu)<+\infty$;
\item $\rho\in W^{1,1}_{\rm loc}(\G_\eps)$ and $\nabla_\eps\rho=w^\eps\rho$ for some $w^\eps\in L^2_\eps(\mu)$.
\end{enumerate}
In this case, $w^\eps\in\tang_\eps(\mu)$ and $|\D_\eps^-\ent|(\mu)=\|w^\eps\|_{L^2_\eps(\mu)}$. Moreover, for any $\nu\in\prob_2(\G_\eps)$, we have 
\begin{equation}\label{eq:HWI_eps}
\ent(\nu)\ge\ent(\mu)-\|w^\eps\|_{L^2_\eps(\mu)}\,\W_\eps(\nu,\mu)-\tfrac{K}{2\eps^2}\,\W_\eps(\nu,\mu)^2,
\end{equation}
where $K>0$ is the constant appearing in~\eqref{eq:ricci_bounded_below}.
\end{proposition} 

\noindent
The equivalence part in \cref{prop:slope_ent_eps} is proved in~\cite{E10}*{Proposition~4.3}. Inequality~\eqref{eq:HWI_eps} is the so-called \emph{HWI inequality} and follows from~\cite{V09}*{Theorem~23.14}, see~\cite{V09}*{Remark~23.16}. 

The quantity 
\begin{equation*}
\F_\eps(\rho)=\|w^\eps\|_{L^2_\eps(\mu)}^2=\int_{\G_\eps\cap\set*{\rho>0}}\frac{\|\nabla_\eps\rho\|_\eps^2}{\rho}\ d\leb^n
\end{equation*}  
appearing in \cref{prop:slope_ent_eps} is the so-called \emph{Fisher information} of $\mu=\rho\leb^n\in\prob_2(\G_\eps)$. The inequality $F_\eps(\rho)\le|\D_\eps^-\ent|(\mu)$ holds in the context of metric measure spaces, see~\cite{AGS14}*{Theorem~7.4}. The converse inequality does not hold in such a generality and heavily depends on the lower semicontinuity of the descending slope $|\D_\eps^-\ent|$, see~\cite{AGS14}*{Theorem~7.6}.

\subsection{The Wasserstein space on the Carnot group}

Let $(\prob_2(\G),\W_\G)$ be the Wasserstein space introduced in \cref{subsec:wasserstein_space} relative to the metric measure space $(\G,\dcc,\leb^n)$. In this section, we discuss the counterparts of \cref{prop:CE_eps} and \cref{prop:slope_ent_eps} in the space $(\prob_2(\G),\W_\G)$. All time-dependent vector fields appearing in the sequel are tacitly understood to be Borel measurable.

Let $\mu\in\prob_2(\G)$ be given. We define the space 
\begin{equation*}
L^2_\G(\mu)=\set*{\xi\in\mathcal{S}(H\G) : \int_{\G}\|\xi\|_\G^2\ d\mu<+\infty}.
\end{equation*}
Here $\mathcal{S}(H\G)$ denotes the set of sections of the horizontal tangent bundle $H\G$. Moreover, we define the \emph{`tangent space'} of $(\prob_2(\G),\W_\G)$ at~$\mu$ as
\begin{equation*}
\tang_\G(\mu)=\overline{\set*{\nabla_\G\phi : \phi\in C^\infty_c(\R^n)}}^{L^2_\G(\mu)}.
\end{equation*}

Given $I\subset\R$ an open interval and a horizontal time-dependent vector field $v^\G\colon I\times\G\to H\G$, $(t,x)\mapsto v^\G_t(x)\in H_x\G$, we say that a curve $(\mu_t)_{t\in I}\subset\prob_2(\G)$ satisfies the continuity equation 
\begin{equation}\label{eq:CE}
\de_t\mu_t+\diverg(v^\G_t\mu_t)=0 \qquad\text{in}\ I\times\G_\eps
\end{equation}
\emph{in the sense of distributions} if 
\begin{equation*}
\int_I\int_\G\|v^\G_t(x)\|_\G\,d\mu_t(x)\,dt<+\infty
\end{equation*}
and
\begin{equation*}
\int_I\int_\G\de_t\phi(t,x)+\scalar*{v^\G_t(x),\nabla_\G\phi(t,x)}_\G\,d\mu_t(x)\,dt=0
\qquad
\forall \phi\in C^\infty_c(I\times\R^n).
\end{equation*} 
The following result is the exact analogue of \cref{prop:CE_eps}. Here and in the sequel, the metric derivative in the Wasserstein space $(\prob_2(\G),\W_\G)$ of a curve $(\mu_t)_{t\in I}\subset\prob_2(\G)$ is denoted by~$|\dot{\mu}_t|_\G$. 

\begin{proposition}[Continuity equation in $(\prob_2(\G),\W_\G)$]\label{prop:CE}
Let $I\subset\R$ be an open interval. If $(\mu_t)_t\in AC^2_{\rm loc}(I;(\prob_2(\G),\W_\G))$, then there exists a horizontal time-dependent vector field $v^\G\colon I\times\G\to H\G$ with $t\mapsto\|v^\G_t\|_{L^2_\G(\mu_t)}\in L^2_{\rm loc}(I)$ such that
\begin{equation}\label{eq:v_G}
v^\G_t\in\tang_\G(\mu_t) 
\qquad
\text{for a.e.}\ t\in I
\end{equation}
and the continuity equation~\eqref{eq:CE} holds in the sense of distributions. The vector field~$v^\G_t$ is uniquely determined in $L^2_\G(\mu_t)$ by~\eqref{eq:CE} and~\eqref{eq:v_G} for a.e.\ $t\in I$ and we have
\begin{equation*}
\|v^\G_t\|_{L^2_\G(\mu_t)}=|\dot{\mu}_t|_\G
\qquad
\text{for a.e.}\ t\in I.
\end{equation*}

Conversely, if $(\mu_t)_{t\in I}\subset\prob_2(\G)$ is a curve satisfying~\eqref{eq:CE} for some $(v^\G_t)_{t\in I}$ such that $t\mapsto\|v^\G_t\|_{L^2_\G(\mu_t)}\in L^2_{\rm loc}(I)$, then $(\mu_t)_t\in AC^2_{\rm loc}(I;(\prob_2(\G),\W_\G))$ with 
\begin{equation*}
|\dot{\mu}_t|_\G\le\|v^\G_t\|_{L^2_\G(\mu_t)}
\qquad
\text{for a.e.}\ t\in I.
\end{equation*}
\end{proposition}

\noindent
As for \cref{prop:CE_eps}, we can interpret the horizontal time-dependent vector field $(v^\G_t)_{t\in I}$ given by \cref{prop:CE} as the `tangent vector' of the curve $(\mu_t)_{t\in I}$ in $(\prob_2(\G),\W_\G)$. An easy adaptation of~\cite{E10}*{Lemma~2.4} to the sub-Riemannian manifold $(\G,\dcc,\leb^n)$ again shows that for a.e.\ $t\in I$ the vector field $v^\G_t$ has minimal $L^2_\G(\mu_t)$-norm among all time-dependent vector fields satisfying~\eqref{eq:CE} and, moreover, that this minimality is equivalent to~\eqref{eq:v_G}.

\cref{prop:CE} can be obtained applying the general results obtained in~\cite{GH15} to the metric measure space $(\G,\dcc,\leb^n)$. Below we give a direct proof exploiting \cref{prop:CE_eps}. The argument is very similar to the one of~\cite{J14}*{Proposition~3.1} and we only sketch it. 

\begin{proof}
If $(\mu_t)_t\in AC^2_{\rm loc}(I;(\prob_2(\G),\W_\G))$, then also $(\mu_t)_t\in AC^2_{\rm loc}(I;(\prob_2(\G_\eps),\W_\eps))$ for every $\eps>0$, since $\deps\le\dcc$. Let $v^\eps\colon I\times\G_\eps\to T\G_\eps$ be the time-dependent vector field given by \cref{prop:CE_eps}. Note that
\begin{equation}\label{eq:equilim_v_eps}
\int_\G\|v^\eps_t\|_\eps^2\ d\mu_t=|\dot{\mu}_t|_\eps^2\le|\dot{\mu}_t|_\G^2
\qquad
\text{for a.e.}\ t\in I.
\end{equation}
Moreover
\begin{equation}\label{eq:norm_v_eps}
\|v^\eps_t\|_\eps^2=\|v^{\eps,V_1}_t\|_1^2+\sum_{i=2}^\kappa\eps^{2(1-i)}\|v^{\eps,V_i}_t\|_1^2
\qquad
\text{for all}\ \eps>0,
\end{equation}
where $v^{\eps,V_i}_t$ denotes the projection of $v^\eps_t$ on~$V_i$. Combining~\eqref{eq:equilim_v_eps} and~\eqref{eq:norm_v_eps}, we find a sequence $(\eps_k)_{k\in\N}$, with $\eps_k\to0$, and a horizontal time-dependent vector field $v^\G\colon I\times\G\to H\G$ such that $v^{\eps_k,V_1}\weakto v^\G$ and $v^{\eps_k,V_i}\to0$ for all $i=2,\dots,\kappa$ as $k\to+\infty$ locally in time in the $L^2$-norm on $I\times\G$ naturally induced by the norm $\|\cdot\|_1$ and the measure $d\mu_tdt$. In particular, $t\mapsto\|v^\G_t\|_{L^2_\G(\mu_t)}\in L^2_{\rm loc}(I)$ and $\|v^\G_t\|_{L^2_\G(\mu_t)}\le|\dot{\mu}_t|_\G$ for a.e.\ $t\in I$. To prove~\eqref{eq:CE}, fix a test function $\phi\in C^\infty_c(I\times\R^n)$ and pass to the limit as $\eps\to0^+$ in~\eqref{eq:CE_eps}.

Conversely, if $(\mu_t)_{t\in I}\subset\prob_2(\G)$ satisfies~\eqref{eq:CE} for some horizontal time-dependent vector field $(v^\G_t)_{t\in I}$ such that $t\mapsto\|v^\G_t\|_{L^2_\eps(\mu_t)}\in L^2_{\rm loc}(I)$, then we can apply \cref{prop:CE_eps} for $\eps=1$. By the \emph{superposition principle} stated in~\cite{B08}*{Theorem~5.8} applied to the Riemannian manifold $(\G_1,\dd_1,\leb^n)$, we find a probability measure $\nu\in\prob(C(I;(\G_1,\dd_1)))$, concentrated on $AC^2_{\rm loc}(I;(\G_1,\dd_1))$, such that $\mu_t=(\mathsf{e}_t)_\#\nu$ for all $t\in I$ and with the property that $\nu$-a.e.\ curve $\gamma\in C(I;(\G_1,\dd_1))$ is an absolutely continuous integral curve of the vector field~$v^\G$. Here $\mathsf{e}_t\colon C(I;(\G_1,\dd_1))\to\G$ denotes the evaluation map at time $t\in I$. Since $v^\G$ is horizontal, $\nu$-a.e.\ curve $\gamma\in C(I;(\G_1,\dd_1))$ is horizontal. Therefore, for all $s,t\in I$, $s<t$, we have 
\begin{equation*}
\dcc(\gamma(t),\gamma(s))\le\int_s^t\|\dot{\gamma}(r)\|_\G\ dr
=\int_s^t\|v^\G_r(\gamma(r))\|_\G\ dr
\end{equation*}
and we can thus estimate 
\begin{align*}
\W_\G^2(\mu_t,\mu_s)&\le\int_{\G\times\G}\dcc^2(x,y)\ d
(\mathsf{e}_t,\mathsf{e}_s)_\#\nu(x,y)
=\int_{AC^2_{\rm loc}}\dcc^2(\gamma(t),\gamma(s))\ d\nu(\gamma)\\
&\le(t-s)\int_{AC^2_{\rm loc}}\int_s^t\|v^\G_r(\gamma(r))\|_\G^2\ dr d\nu(\gamma)
=(t-s)\int_s^t\int_\G\|v^\G_r\|_\G^2\ d\mu_r dr.
\end{align*} 
This immediately gives $|\dot{\mu}_t|_\G\le\|v^\G_t\|_{L^2_\G(\mu_t)}$ for a.e.\ $t\in I$, which in turn proves~\eqref{eq:v_G}. 
\end{proof}

To establish an analogue of \cref{prop:slope_ent_eps}, we need to prove the two inequalities separately. For $\mu=\rho\leb^n\in\prob_2(\G)$, the inequality $\F_\G(\rho)\le|\D^-_\G\ent|(\mu)$ is stated in \cref{prop:slope_ent_G_general} below. Here and in the sequel, $|\D_\G^-\ent|(\mu)$ denotes the descending slope of the entropy~$\ent$ at the point $\mu\in\prob_2(\G)$ in the Wasserstein space $(\prob_2(\G),\W_\G)$. 

\begin{proposition}\label{prop:slope_ent_G_general}
Let $\mu=\rho\leb^n\in\prob_2(\G)$. If $|\D^-_\G\ent|(\mu)<+\infty$, then $\rho\in W^{1,1}_{\G,\,\rm loc}(\G)$ and $\nabla_\G\rho=w^\G\rho$ for some horizontal vector field $w^\G\in L^2(\mu)$ with $\|w^\G\|_{L^2_\G(\mu)}\le|\D^-_\G\ent|(\mu)$.
\end{proposition} 

\noindent
\cref{prop:slope_ent_G_general} can be obtained by applying~\cite{AGS14}*{Theorem~7.4} to the metric measure space $(\G,\dcc,\leb^n)$. Below we give a direct proof of this result which is closer in the spirit to the one in the Riemannian setting, see~\cite{E10}*{Lemma~4.2}. See also~\cite{J14}*{Proposition~3.1}.

\begin{proof}
Let $V\in C^\infty_c(\G;H\G)$ be a smooth horizontal vector field with compact support. Then there exists $\delta>0$ such that, for any $t\in(-\delta,\delta)$, the flow map of the vector field~$V$ at time~$t$, namely
\begin{equation*}
F_t(x):=\exp_x(t V), \qquad x\in\G, 
\end{equation*}
is a diffeomorphism and $J_t=\det(D F_t)$ is such that $c^{-1}\le J_t\le c$ for some $c\ge1$. By the change of variable formula, the measure $\mu_t:=(F_t)_\#\mu$ is such that $\mu_t=\rho_t\leb^n$ with $J_t\rho_t=\rho\circ F_t^{-1}$ for $t\in(-\delta,\delta)$. Let us set $H(r)=r\log r$ for $r\ge0$. Then, for $t\in(-\delta,\delta)$,
\begin{equation*}
\ent(\mu_t)=\int_\G H(\rho_t)\ dx=\int_\G H\left(\frac{\rho}{J_t}\right) J_t\ dx=\ent(\mu)-\int_\G \rho\log(J_t)\ dx<+\infty.
\end{equation*}
Note that $J_0=1$, $\dot{J}_0=\diverg V$ and that $t\mapsto\dot{J}_t J_t^{-1}$ is uniformly bounded for $t\in(-\delta,\delta)$. Thus we have
\begin{equation*}
\frac{d}{dt}\ent(\mu_t)\bigg|_{t=0}
=-\frac{d}{dt}\int_\G \rho\log(J_t)\ dx\bigg|_{t=0}\ dx
=\int_\G -\rho\,\frac{\dot{J}_t}{J_t}\bigg|_{t=0}\ dx
=-\int_\G\rho\diverg V\ dx.
\end{equation*}
On the other hand, we have
\begin{equation*}
\W_\G^2(\mu_t,\mu)=\W_\G^2((F_t)_\#\mu,\mu)\le\int_\G\dcc^2(F_t(x),x)\ d\mu(x) 
\end{equation*}
and so
\begin{equation*}
\limsup_{t\to0}\frac{\W_\G^2(\mu_t,\mu)}{|t|^2}
\le\int_\G\limsup_{t\to0}\frac{\dcc^2(F_t(x),x)}{|t|^2}\ d\mu(x)
=\int_\G\|V\|_\G^2\ d\mu.
\end{equation*}
Hence
\begin{align*}
-\frac{d}{dt}\ent(\mu_t)\bigg|_{t=0}\le\limsup_{t\to0}\frac{[\ent(\mu_t)-\ent(\mu)]^-}{\W_\G(\mu_t,\mu)}\cdot\frac{\W_\G(\mu_t,\mu)}{|t|}\le|\D^-_\G\ent|(\mu)\left(\int_\G\|V\|_\G^2\ d\mu\right)^\frac{1}{2}
\end{align*}
and thus
\begin{equation*}
\left|\int_\G\rho\diverg V\ dx\right|\le|\D^-_\G\ent|(\mu)\left(\int_\G\|V\|_\G^2\ d\mu\right)^\frac{1}{2}.
\end{equation*}
By Riesz representation theorem, we conclude that there exists a horizontal vector field $w^\G\in L^2_\G(\mu)$ such that $\|w_\G\|_{L^2_\G(\mu)}\le|\D^-_\G\ent|(\mu)$ and
\begin{equation*}
-\int_\G\rho\diverg V\ dx=\int_\G\scalar*{w^\G,V}_\G\ d\mu \qquad \text{for all}\ V\in C^\infty_c(\G;H\G).
\end{equation*}
This implies that $\nabla_\G\rho=w^\G\rho$ and the proof is complete.
\end{proof}

We call the quantity 
\begin{equation*}
\F_\G(\rho)=\|w^\G\|_{L^2_\G(\mu)}^2=\int_{\G\cap\set*{\rho>0}}\frac{\|\nabla_\G\rho\|_\G^2}{\rho}\ d\leb^n
\end{equation*}  
appearing in \cref{prop:slope_ent_G_general} the \emph{horizontal Fisher information} of $\mu=\rho\leb^n\in\prob_2(\G)$. On its effective domain, $\F_\G$ is convex and sequentially lower semicontinuous with respect to the weak topology of $L^1(\G)$, see~\cite{AGS14}*{Lemma~4.10}. 

Given $\mu=\rho\leb^n\in\prob_2(\G)$, it is not clear how to prove the inequality $|\D^-_\G\ent|^2(\mu)\le\F_\G(\rho)$ under the mere condition $|\D^-_\G\ent|(\mu)<+\infty$. Following~\cite{J14}*{Proposition~3.4}, in \cref{prop:slope_ent_G_special} below we show that the condition $|\D^-_\eps\ent|(\mu)<+\infty$ for some $\eps>0$ (and thus any) implies that $|\D^-_\G\ent|^2(\mu)\le\F_\G(\rho)$.

\begin{proposition}\label{prop:slope_ent_G_special}
Let $\mu=\rho\leb^n\in\prob_2(\G)$. If $|\D_\eps^-\ent|(\mu)<+\infty$ for some $\eps>0$, then also $|\D^-_\G\ent|(\mu)<+\infty$ and moreover $\F_\G(\rho)=|\D^-_\G\ent|^2(\mu)$.
\end{proposition}

\begin{proof}
Since $|\D_\eps^-\ent|(\mu)<+\infty$, we have $\ent(\mu)<+\infty$. Since $\deps\le\dcc$ and so $\W_\eps\le\W_\G$, we also have $|\D^-_\G\ent|(\mu)\le|\D_\eps^-\ent|(\mu)$. By \cref{prop:slope_ent_G_general}, we conclude that $\rho\in W^{1,1}_{\G,\,\rm loc}(\G)$ and $\nabla_\G\rho=w^\G\rho$ for some horizontal vector field $w^\G\in L^2_\G(\mu)$ with $\|w^\G\|_{L^2_\G(\mu)}\le|\D^-_\G\ent|(\mu)$. We now prove the converse inequality. Since $|\D_\eps^-\ent|(\mu)<+\infty$, by \cref{prop:slope_ent_eps} we have $\F_\eps(\rho)=|\D_\eps^-\ent|^2(\mu)$ and
\begin{equation}\label{eq:hwi_ineq_eps}
\ent(\nu)\ge\ent(\mu)-\F^{1/2}_\eps(\rho)\,\W_\eps(\nu,\mu)-\tfrac{K}{2\eps^2}\,\W_\eps^2(\nu,\mu)
\end{equation}
for any $\nu\in\prob_2(\G)$. Take $\eps=\W_\G(\nu,\mu)^{1/4}$ and assume $\eps<1$. Since $\W_\eps\le\W_\G$, from~\eqref{eq:hwi_ineq_eps} we get
\begin{align}\label{eq:calo_hwi_eps}
\ent(\nu)&\ge\ent(\mu)-\F^{1/2}_\eps(\rho)\,\W_\G(\nu,\mu)-\tfrac{K}{2\eps^2}\,\W_\G^2(\nu,\mu)\nonumber\\
&=\ent(\mu)-\F^{1/2}_\eps(\rho)\,\W_\G(\nu,\mu)-\tfrac{K}{2}\,\W_\G^{3/2}(\nu,\mu).
\end{align}
We need to bound $\F_\eps(\rho)$ from above in terms of $\F_\G(\rho)$. To do so, observe that
\begin{equation*}
\nabla_\eps\rho=\nabla_\G\rho+\sum_{i=2}^k\eps^{2(i-1)}\nabla_{V_i}\rho,\qquad
\|\nabla_\eps\rho\|_\eps^2=\|\nabla_\G\rho\|_\G^2+\sum_{i=2}^k\eps^{2(i-1)}\|\nabla_{V_i}\rho\|_\G^2.
\end{equation*}
In particular, $\tfrac{\nabla_{V_i}\rho}{\rho}\in L^2_\G(\mu)$ for all $i=2,\dots,k$. Recalling the inequality $(1+r)\le\left(1+\frac{r}{2}\right)^2$ for $r\ge0$, we can estimate
\begin{align*}
\F_\eps(\rho)&=\F_\G(\rho)+\sum_{i=2}^k\eps^{2(i-1)}\left\|\tfrac{\nabla_{V_i}\rho}{\rho}\right\|_{L^2_\G(\mu)}^2
=\F_\G(\rho)\,\left(1+\frac{1}{\F_\G(\rho)}\sum_{i=2}^k\eps^{2(i-1)}\left\|\tfrac{\nabla_{V_i}\rho}{\rho}\right\|_{L^2_\G(\mu)}^2\right)\\
&\le\F_\G(\rho)\,\left(1+\frac{1}{2\F_\G(\rho)}\sum_{i=2}^k\eps^{2(i-1)}\left\|\tfrac{\nabla_{V_i}\rho}{\rho}\right\|_{L^2_\G(\mu)}^2\right)^2
\end{align*}
and thus
\begin{equation}\label{eq:estimate_fisher_fisher}
\F_\eps^{1/2}(\rho)\le\F_\G^{1/2}(\rho)\,\left(1+\frac{1}{2\F_\G(\rho)}\sum_{i=2}^k\eps^{2(i-1)}\left\|\tfrac{\nabla_{V_i}\rho}{\rho}\right\|_{L^2_\G(\mu)}^2\right).
\end{equation} 
Inserting~\eqref{eq:estimate_fisher_fisher} into~\eqref{eq:calo_hwi_eps}, we finally get
\begin{equation*}
\ent(\nu)\ge\ent(\mu)-\F^{1/2}_\G(\rho)\,\W_\G(\nu,\mu)-C\,\W_\G^{3/2}(\nu,\mu)
\end{equation*}
for some $C>0$ independent of~$\eps$. This immediately leads to $|\D^-_\G\ent|(\mu)\le\F^{1/2}_\G(\rho)$.
\end{proof}

\subsection{Carnot groups are non-\texorpdfstring{$CD(K,\infty)$}{CD(K,infty)} spaces}

As stated in~\cite{AGS14}*{Theorem~7.6}, if the metric measure space $(X,\dd,\mathfrak{m})$ is Polish and satisfies~\eqref{eq:assumption_on_measure}, then the properties
\begin{enumerate}[(i)]
\item $|\D^-\ent|^2(\mu)=\F(\rho)$ for all $\mu=\rho\mathfrak{m}\in\dom(\ent)$;
\item\label{item:lsc_slope} $|\D^-\ent|$ is sequentially lower semicontinuous with respect to convergence with moments in~$\prob(X)$ on sublevels of~$\ent$;
\end{enumerate}
are equivalent. We do not know if property~\eqref{item:lsc_slope} is true for the space $(\G,\dcc,\leb^n)$ and this is why in \cref{prop:slope_ent_G_special} we needed the additional assumption $|\D^-_\eps\ent|(\mu)<+\infty$. 

By~\cite{AGS14}*{Theorem~9.3}, property~\eqref{item:lsc_slope} holds true if $(X,\dd,\mathfrak{m})$ is $CD(K,\infty)$ for some $K\in\R$. As the following result shows, (non-commutative) Carnot groups are not $CD(K,\infty)$, so that the validity of property~\eqref{item:lsc_slope} in these metric measure spaces is an open problem. Note that \cref{prop:carnot_not_CD} below was already known for the Heisenberg groups, see~\cite{J09}.

\begin{proposition}\label{prop:carnot_not_CD}
If $(\G,\dcc,\leb^n)$ is a non-commutative Carnot group, then the metric measure space $(\G,\dcc,\leb^n)$ is not $CD(K,\infty)$ for any $K\in\R$.
\end{proposition}

\begin{proof}
By contradiction, assume that $(\G,\dcc,\leb^n)$ is a $CD(K,\infty)$ space for some $K\in\R$. Since the Dirichlet--Cheeger energy associate to the horizontal gradient is \emph{quadratic} on $L^2(\G,\leb^n)$ (see~\cite{AGS14-2}*{Section~4.3} for a definition), by~\cite{AGMR15}*{Theorem~6.1} we deduce that $(\G,\dcc,\leb^n)$ is a ($\sigma$-finite) $RCD(K,\infty)$ space. By~\cite{AGMR15}*{Theorem~7.2}, we deduce that $(\G,\dcc,\leb^n)$ satisfies the $BE(K,\infty)$ property, that is,
\begin{equation}\label{eq:BE}
\|\nabla_\G (P_t f)\|_\G^2\le e^{-2Kt} P_t(\|\nabla f\|_\G^2),
\qquad
\text{for all }
t\ge 0,\ f\in C^\infty_c(\R^n).
\end{equation}
Here and in the rest of the proof, we set $P_t f:=f\star\h_t$ for short. Arguing similarly as in the proof of~\cite{W11}*{Theorem~1.1}, it is possible to prove that~\eqref{eq:BE} is equivalent to the following \emph{reverse Poincaré inequality}
\begin{equation}\label{eq:reverse_P}
P_t(f^2)-(P_t f)^2\ge 2I_{2K}(t)\,\|\nabla (P_t f)\|_\G^2,
\qquad
\text{for all }
t\ge 0,\ f\in C^\infty_c(\R^n),
\end{equation}
where $I_K(t):=\frac{e^{Kt}-1}{K}$ if $K\ne0$ and $I_0(t):=t$. Now, by~\cite{BB16}*{Propositions~2.5 and~2.6}, there exists a constant $\Lambda\in\left[\frac{Q}{2m_1},\frac{Q}{m_1}\right]$ (where~$Q$ and~$m_1$ are as in \cref{subsec:carnot_groups}) such that the inequality
\begin{equation}\label{eq:reverse_P_sharp}
P_t(f^2)-(P_t f)^2\ge\frac{t}{\Lambda}\,\|\nabla (P_t f)\|_\G^2,
\qquad
\text{for all }
t\ge 0,\ f\in C^\infty_c(\R^n),
\end{equation}
holds true and, moreover, is sharp. Comparing~\eqref{eq:reverse_P} and~\eqref{eq:reverse_P_sharp}, we thus must have that $\Lambda\le\frac{t}{2I_{2K}(t)}$ for all $t>0$. Passing to the limit as $t\to0^+$, we get that $\Lambda\le\frac{1}{2}$, so that $Q\le m_1$. This immediately implies that~$\G$ is commutative, a contradiction.  
\end{proof}

\section{Proof of the main result}
\label{sec:proof_of_main_result}

\subsection{Heat diffusions are gradient flows of the entropy}

In this section we prove the first part of \cref{th:main}. The argument follows the strategy outlined in~\cite{J14}*{Section~4.1}. 

The following technical lemma will be applied to horizontal vector fields in the proof of \cref{prop:entropy_dissipation} below. The proof is exactly the same of~\cite{J14}*{Lemma~4.1} and we omit it.

\begin{lemma}\label{lemma:divergence}
Let $V\colon\R^n\to\R^n$ be a vector field with locally Lipschitz coefficients such that $|V|_{\R^n}\in L^1(\R^n)$ and $\diverg V\in L^1(\R^n)$. Then $\int_{\R^n}\diverg V\ dx=0$.
\end{lemma}

\cref{prop:entropy_dissipation} below states that the function $t\mapsto\ent(\rho_t\leb^n)$ is locally absolutely continuous if $(\rho_t)_{t\ge0}$ solves the sub-elliptic heat equation~\eqref{eq:heat_diffusion} with initial datum $\rho_0\in L^1(\G)$ such that $\mu_0=\rho_0\leb^n\in\prob_2(\G)$. By~\cite{AGS14}*{Proposition~4.22}, this result is true under the stronger assumption that $\rho_0\in L^1(\G)\cap L^2(\G)$. Here the point is to remove the $L^2$-integrability condition on the initial datum exploiting the estimates on the heat kernel collected in \cref{th:property_heat_kernel}, see also~\cite{J14}*{Section~4.1.1}.  

\begin{proposition}[Entropy dissipation]\label{prop:entropy_dissipation}
Let $\rho_0\in L^1(\G)$ be such that $\mu_0=\rho_0\leb^n\in\dom(\ent)$. If $(\rho_t)_{t\ge0}$ solves the sub-elliptic heat equation $\de_t\rho_t=\Delta_\G\rho_t$ with initial datum~$\rho_0$, then the map $t\mapsto\ent(\mu_t)$, $\mu_t=\rho_t\leb^n$, is locally absolutely continuous on $(0,+\infty)$ and it holds 
\begin{equation}\label{eq:entropy_dissipation}
\frac{d}{dt}\ent(\mu_t)=-\int_{\G\cap\set*{\rho_t>0}}\frac{\|\nabla_\G\rho_t\|^2_\G}{\rho_t}\ dx
\qquad
\text{for a.e.}\ t>0.
\end{equation}
\end{proposition}

\begin{proof}
Note that $\mu_t\in\prob_2(\G)$ for all $t>0$ by~\eqref{eq:heat_second_moment}. Hence $\ent(\mu_t)>-\infty$ for all $t>0$. Since $C_t:=\sup_{x\in\G}\h_t(x)<+\infty$ for each fixed $t>0$ by~\eqref{eq:heat_estimate_above}, we get that $\rho_t\le C_t$ for all $t>0$. Thus $\ent(\mu_t)<+\infty$ for all $t>0$.

For each $m\in\N$, define 
\begin{equation}\label{eq:def_truncated_entropy}
z_m(r):=\min\set*{m,\max\set*{1+\log r,-m}}, \qquad H_m(r):=\int_0^r z_m(s)\ ds, \qquad r\ge0.
\end{equation}
Note that $H_m$ is of class $C^1$ on $[0,+\infty)$ with $H_m'$ is globally Lipschitz and bounded. We claim that
\begin{equation}\label{eq:deriv_1}
\frac{d}{dt}\int_\G H_m(\rho_t)\ dx=\int_\G z_m(\rho_t)\Delta_\G\rho_t\ dx
\qquad
\forall t>0,\ \forall m\in\N.
\end{equation}
Indeed, we have $|H_m(\rho_t)|\le m\rho_t\in L^1(\G)$ and, given $[a,b]\subset(0,+\infty)$, by~\eqref{eq:heat_estimate_derivatives} the function $x\mapsto\sup_{t\in[a,b]}|\Delta_\G\h_t(x)|$ is bounded. Thus
\begin{equation*}
\sup_{t\in[a,b]}\left|\frac{d}{dt}H_m(\rho_t)\right|\le m\sup_{t\in[a,b]}\left(\rho_0\star|\Delta_\G\h_t|\right)\le m\,\rho_0\star\sup_{t\in[a,b]}|\Delta_\G\h_t|\in L^1(\G).
\end{equation*}
Therefore~\eqref{eq:deriv_1} follows by differentiation under integral sign. We now claim that
\begin{equation}\label{eq:deriv_2}
\int_\G z_m(\rho_t)\Delta_\G\rho_t\ dx=-\int_{\set*{e^{-m-1}<\rho_t<e^{m-1}}} \frac{\|\nabla_\G\rho_t\|_\G^2}{\rho_t}\ dx
\qquad
\forall t>0,\ \forall m\in\N.
\end{equation}
Indeed, by Cauchy--Schwarz inequality, we have
\begin{equation}\label{eq:C-S_for_horizontal_grad}
\begin{split}
\frac{\|\nabla_\G\rho_t(x)\|_\G^2}{\rho_t(x)}
&\le\frac{\left[(\rho_0\star\|\nabla_\G\h_t\|_\G)(x)\right]^2}{(\rho_0\star\h_t)(x)}\\
&=\frac{1}{(\rho_0\star\h_t)(x)}\left[\int_\G\sqrt{\rho_0(xy^{-1})}\,\frac{\|\nabla_\G\h_t(y)\|_\G}{\sqrt{\h_t(y)}}\cdot\sqrt{\rho_0(xy^{-1})}\,\sqrt{\h_t(y)}\ dy\right]^2\\
&\le\frac{1}{(\rho_0\star\h_t)(x)}
	\left(\int_\G\rho_0(xy^{-1})\,\frac{\|\nabla_\G\h_t(y)\|_\G^2}{\h_t(y)}\ dy\right)
	\left(\int_\G\rho_0(xy^{-1})\,\h_t(y)\ dy\right)\\
&\le\left(\rho_0\star\frac{\|\nabla_\G\h_t\|_\G^2}{\h_t}\right)(x)
\qquad
\text{for all}\ x\in\G.
\end{split}
\end{equation}
Thus, by~\eqref{eq:heat_estimate_below} and~\eqref{eq:heat_estimate_derivatives}, we get
\begin{equation}\label{eq:go_to_h}
\int_{\set*{e^{-m-1}<\rho_t<e^{m-1}}}\frac{\|\nabla_\G\rho_t\|_\G^2}{\rho_t}\ dx\le\int_\G\rho_0\star\frac{\|\nabla_\G\h_t\|_\G^2}{\h_t}\ dx=\int_\G\frac{\|\nabla_\G\h_t\|_\G^2}{\h_t}\ dx<+\infty.
\end{equation}
This, together with~\eqref{eq:deriv_1}, proves that
\begin{align}\label{eq:diverg_L1}
\diverg(z_m(\rho_t)\nabla_\G\rho_t)=z_m'(\rho_t)\|\nabla_\G\rho_t\|_\G^2-z_m(\rho_t)\Delta_\G\rho_t\in L^1(\G).
\end{align}
Thus~\eqref{eq:deriv_2} follows by integration by parts provided that
\begin{equation}\label{eq:diverg_0}
\int_\G\diverg(z_m(\rho_t)\nabla_\G\rho_t)\ dx=0.
\end{equation}
To prove~\eqref{eq:diverg_0}, we apply \cref{lemma:divergence} to the vector field $V=z_m(\rho_t)\nabla_\G\rho_t$. By~\eqref{eq:diverg_L1}, we already know that $\diverg V\in L^1(\G)$, so we just need to prove that $|V|_{\R^n}\in L^1(\G)$. Note that
\begin{equation*}
\int_\G|V|_{\R^n}\ dx\le m\int_\G\rho_0\star|\nabla_\G\h_t|_{\R^n}\ dx=m\int_\G |\nabla_\G\h_t|_{\R^n}\ dx,
\end{equation*}
so it is enough to prove that $|\nabla_\G\h_t|_{\R^n}\in L^1(\G)$. But we have
\begin{equation*}
|\nabla_\G\h_t(x)|_{\R^n}\le p(x_1,\dots,x_n)\|\nabla_\G\h_t(x)\|_\G, \qquad x\in\G,
\end{equation*}
where $p\colon\R^n\to[0,+\infty)$ is a function with polynomial growth, because the horizontal vector fields $X_1,\dots,X_{h_1}$ have polynomial coefficients. Since $\dcc$ is equivalent to $\dd_\infty$, where~$\dd_\infty$ was introduced in~\eqref{eq:def_box_norm}, by~\eqref{eq:heat_estimate_derivatives} we conclude that $|\nabla_\G\h_t|_{\R^n}\in L^1(\G)$. This completes the proof of~\eqref{eq:diverg_0}. 

Combining~\eqref{eq:deriv_1} and~\eqref{eq:deriv_2}, we thus get
\begin{equation*}
\frac{d}{dt}\int_\G H_m(\rho_t)\ dx=-\int_{\set*{e^{-m-1}<\rho_t<e^{m-1}}} \frac{\|\nabla_\G\rho_t\|_\G^2}{\rho_t}\ dx
\qquad
\forall t>0,\ \forall m\in\N.
\end{equation*}
Note that
\begin{equation}\label{eq:fisher_L1_loc}
t\mapsto\int_\G \frac{\|\nabla_\G\rho_t\|_\G^2}{\rho_t}\ dx\in L^1_{\rm loc}(0,+\infty).
\end{equation}
Indeed, from~\eqref{eq:heat_estimate_below} and~\eqref{eq:heat_estimate_derivatives} we deduce that
\begin{equation*}
t\mapsto\int_\G \frac{\|\nabla_\G\h_t\|_\G^2}{\h_t}\ dx\in L^1_{\rm loc}(0,+\infty).
\end{equation*}
Recalling~\eqref{eq:C-S_for_horizontal_grad} and~\eqref{eq:go_to_h}, this immediately implies~\eqref{eq:fisher_L1_loc}. Therefore
\begin{equation*}
\int_\G H_m(\rho_{t_1})\ dx-\int_\G H_m(\rho_{t_0})\ dx=-\int_{t_0}^{t_1}\int_{\set*{e^{-m-1}<\rho_t<e^{m-1}}} \frac{\|\nabla_\G\rho_t\|_\G^2}{\rho_t}\ dx dt
\end{equation*}
for any $t_0,t_1\in(0,+\infty)$ with $t_0<t_1$ and $m\in\N$. We now pass to the limit as $m\to+\infty$. Note that $H_m(r)\to r\log r$ as $m\to+\infty$ and that for all $m\in\N$
\begin{equation}\label{eq:truncated_ent_prop_1}
r\log r\le H_{m+1}(r)\le H_m(r) \qquad\text{for}\ r\in[0,1]
\end{equation}
and
\begin{equation}\label{eq:truncated_ent_prop_2}
0\le H_m(r)\le 1+r\log r \qquad\text{for}\ r\in[1,+\infty).
\end{equation}
Thus
\begin{equation*}
\lim_{m\to+\infty}\int_\G H_m(\rho_t)\ dx=\int_\G \rho_t\log\rho_t\ dx
\end{equation*}
by the monotone convergence theorem on $\set*{\rho_t\le1}$ and by the dominated convergence theorem on $\set*{\rho_t>1}$. Moreover  
\begin{equation*}
\lim_{m\to+\infty}\int_{t_0}^{t_1}\int_{\set*{e^{-m-1}<\rho_t<e^{m-1}}} \frac{\|\nabla_\G\rho_t\|_\G^2}{\rho_t}\ dx dt
=\int_{t_0}^{t_1}\int_{\G\cap\set*{\rho_t>0}} \frac{\|\nabla_\G\rho_t\|_\G^2}{\rho_t}\ dx dt
\end{equation*}
by the monotone convergence theorem. This concludes the proof.
\end{proof}

We are now ready to prove the first part of \cref{th:main}. The argument follows the strategy outlined in~\cite{J14}*{Section~4.1}. See also the first part of the proof of~\cite{AGS14}*{Theorem~8.5}.

\begin{theorem}\label{th:heat_diff_implies_GF_ent}
Let $\rho_0\in L^1(\G)$ be such that $\mu_0=\rho_0\leb^n\in\dom(\ent)$. If $(\rho_t)_{t\ge0}$ solves the sub-elliptic heat equation $\de_t\rho_t=\Delta_\G\rho_t$ with initial datum $\rho_0$, then $\mu_t=\rho_t\leb^n$ is a gradient flow of $\ent$ in $(\prob_2(\G),\W_\G)$ starting from~$\mu_0$.
\end{theorem}

\begin{proof}
Note that $(\mu_t)_{t>0}\subset\prob_2(\G)$, see the proof of \cref{prop:entropy_dissipation}. Moreover, $(\mu_t)_{t>0}$ satisfies~\eqref{eq:CE} with $v^\G_t=\nabla_\G\rho_t/\rho_t$ for $t>0$. Note that $t\mapsto\|v^\G_t\|_{L^2_\G(\mu_t)}\in L^2_{\rm loc}(0,+\infty)$ by~\eqref{eq:heat_estimate_below}, \eqref{eq:heat_estimate_derivatives}, \eqref{eq:C-S_for_horizontal_grad} and~\eqref{eq:go_to_h}. By \cref{prop:CE} we conclude that
\begin{equation}\label{eq:estim_deriv_heat_curve}
|\dot{\mu}_t|^2\le\int_{\G\cap\set*{\rho_t>0}}\frac{\|\nabla\rho_t\|_\G^2}{\rho_t}\ dx
\qquad 
\text{for a.e.}\ t>0.
\end{equation}
By \cref{prop:entropy_dissipation}, the map $t\mapsto\ent(\mu_t)$ is locally absolutely continuous on $(0,+\infty)$ and so, by the chain rule, we get
\begin{equation}\label{eq:chain_rule_ent}
-\frac{d}{dt}\ent(\mu_t)\le|\D^-_\G\ent|(\mu_t)\cdot|\dot{\mu}_t|_\G \qquad 
\text{for a.e.}\ t>0.
\end{equation}
Thus, if we prove that
\begin{equation}\label{eq:slope_ent_finite_heat}
|\D^-_\G\ent|^2(\mu_t)=\int_{\G\cap\set*{\rho_t>0}}\frac{\|\nabla_\G\rho_t\|_\G^2}{\rho_t}\ dx
\qquad 
\text{for a.e.}\ t>0
\end{equation}
then, combining this equality with~\eqref{eq:entropy_dissipation}, \eqref{eq:estim_deriv_heat_curve} and~\eqref{eq:chain_rule_ent}, we find that
\begin{equation*}
|\dot{\mu}_t|_\G=|\D^-_\G\ent|(\mu_t),
\qquad
\frac{d}{dt}\ent(\mu_t)=-|\D^-_\G\ent|(\mu_t)\cdot|\dot{\mu}_t|_\G
\qquad 
\text{for a.e.}\ t>0,
\end{equation*}
so that $(\mu_t)_{t\ge0}$ is a gradient flow of~$\ent$ starting from~$\mu_0$ as observed in \cref{rem:GF_AC}.

We now prove~\eqref{eq:slope_ent_finite_heat}. To do so, we apply \cref{prop:slope_ent_G_special}. We need to check that $|\D^-_\eps\ent|(\mu_t)<+\infty$ for some $\eps>0$. To prove this, we apply \cref{prop:slope_ent_eps}. Since $\deps\le\dcc$, we have $\mu_t\in\prob_2(\G_\eps)$ for all $\eps>0$. Moreover
\begin{equation*}
\F_\eps(\rho_t)=\F_\G(\rho_t)+\sum_{i=2}^k\eps^{2(i-1)}\int_{\G\cap\set*{\rho_t>0}}\frac{\|\nabla_{V_i}\rho_t\|_\G^2}{\rho_t}\ dx.
\end{equation*}
Since $\F_\G(\rho_t)<+\infty$, we just need to prove that
\begin{equation*}
\int_{\G\cap\set*{\rho_t>0}}\frac{\|\nabla_{V_i}\rho_t\|_\G^2}{\rho_t}\ dx<+\infty
\end{equation*}
for all $i=2,\dots,\kappa$. Indeed, arguing as in~\eqref{eq:C-S_for_horizontal_grad}, by Cauchy--Schwarz inequality we have
\begin{align*}
\frac{\|\nabla_{V_i}\rho_t\|_\G^2}{\rho_t}
\le\frac{\left(\rho\star\|\nabla_{V_i}\h_t\|_\G\right)^2}{\rho\star\h_t}
=\frac{1}{\rho\star\h_t}\left[\rho\star\left(\frac{\|\nabla_{V_i}\h_t\|_\G}{\sqrt{\h_t}}\,\sqrt{\h_t}\right)\right]^2
\le\rho\star\frac{\|\nabla_{V_i}\h_t\|_\G^2}{\h_t}.
\end{align*}
Therefore, by~\eqref{eq:heat_estimate_below} and~\eqref{eq:heat_estimate_derivatives}, we get
\begin{equation*}
\int_{\G\cap\set*{\rho_t>0}}\frac{\|\nabla_{V_i}\rho_t\|_\G^2}{\rho_t}\ dx
\le\int_\G\rho\star\frac{\|\nabla_{V_i}\h_t\|_\G^2}{\h_t}\ dx
=\int_\G\frac{\|\nabla_{V_i}\h_t\|_\G^2}{\h_t}\ dx<+\infty.
\end{equation*}
This concludes the proof. 
\end{proof}

\subsection{Gradient flows of the entropy are heat diffusions}

In this section we prove the second part of \cref{th:main}. Our argument is different from the one presented in~\cite{J14}*{Section~4.2}. However, as observed in~\cite{J14}*{Remark~5.3}, the techniques developed in~\cite{J14}*{Section~4.2} can be adapted in order to obtain a proof of \cref{th:GF_ent_implies_heat_diff} below for any Carnot group~$\G$ of step~$2$.

Let us start with the following remark. If $(\mu_t)_{t\ge0}$ is a gradient flow of $\ent$ in $(\prob_2(\G),\W_\G)$ then, recalling \cref{def:metric_GF}, we have that $\mu_t\in\dom(\ent)$ for all $t\ge0$. By~\eqref{eq:def_entropy}, this means that $\mu_t=\rho_t\leb^n$ for some probability density $\rho_t\in L^1(\G)$ for all $t\ge0$. In addition, $t\mapsto|\D_\G^-\ent|(\mu_t)\in L^2_{\rm loc}([0,+\infty))$ and the function $t\mapsto\ent(\mu_t)$ is non-increasing, therefore a.e.\ differentiable and locally integrable on~$[0,+\infty)$. 

\cref{lemma:GF_ent_plus_AC_is_heat_diff} below shows that it is enough to establish~\eqref{eq:GF_ent_plus_AC} in order to prove the second part of \cref{th:main}. For the proof, see also the last paragraph of~\cite{J14}*{Section~4.2}.

\begin{lemma}\label{lemma:GF_ent_plus_AC_is_heat_diff}
Let $\rho_0\in L^1(\G)$ be such that $\mu_0=\rho_0\leb^n\in\dom(\ent)$. Assume $(\mu_t)_{t\ge0}$ is a gradient flow of $\ent$ in $(\prob_2(\G),\W_\G)$ starting from~$\mu_0$, with $\mu_t=\rho_t\leb^n$ for all $t\ge0$. Let $(v^\G_t)_{t>0}$ and $(w^\G_t)_{t>0}$, $w^\G_t=\nabla_\G\rho_t/\rho_t$, be the horizontal time-dependent vector fields given by \cref{prop:CE} and \cref{prop:slope_ent_G_general} respectively. If it holds
\begin{equation}\label{eq:GF_ent_plus_AC}
-\frac{d}{dt}\ent(\mu_t)\le\int_\G\scalar*{-w^\G_t,v^\G_t}_\G\ d\mu_t 
\qquad
\text{for a.e.}\ t>0,
\end{equation}
then $(\rho_t)_{t\ge0}$ solves the sub-elliptic heat equation $\de_t\rho_t=\Delta_\G\rho_t$ with initial datum $\rho_0$.
\end{lemma}

\begin{proof}
From \cref{def:metric_GF} we get that
\begin{equation*}
\ent(\mu_t)+\frac{1}{2}\int_s^t|\dot{\mu}_r|_\G^2\ dr+\frac{1}{2}\int_s^t|\D^-\ent|_\G^2(\mu_r)\ dr\le\ent(\mu_s)
\end{equation*}
for all $s,t\ge0$ with $s\le t$. Therefore 
\begin{equation*}
-\frac{d}{dt}\ent(\mu_t)\ge\frac{1}{2}|\dot{\mu}_t|_\G^2+\frac{1}{2}|\D^-\ent|^2_\G(\mu_t)
\qquad 
\text{for a.e.}\ t>0.
\end{equation*}
By Young's inequality, \cref{prop:CE} and \cref{prop:slope_ent_G_general}, we thus get
\begin{equation}\label{eq:GF_ent_plus_AC_opposite}
-\frac{d}{dt}\ent(\mu_t)\ge|\dot{\mu}_t|_\G\cdot|\D^-_\G\ent|(\mu_t)\ge\|v^\G_t\|_{L^2_\G(\mu_t)}\cdot\|w^\G_t\|_{L^2_\G(\mu_t)}
\qquad 
\text{for a.e.}\ t>0.
\end{equation}
Combining~\eqref{eq:GF_ent_plus_AC} and~\eqref{eq:GF_ent_plus_AC_opposite}, by Cauchy--Schwarz inequality we conclude that $v^\G_t=-w^\G_t=-\nabla_\G\rho_t/\rho_t$ in $L^2_\G(\mu_t)$ for a.e.\ $t>0$. This immediately implies that $(\rho_t)_{t\ge0}$ solves the sub-elliptic heat equation $\de_t\rho_t=\Delta_\G\rho_t$ with initial datum $\rho_0$ \emph{in the sense of distributions}, i.e.\
\begin{equation*}
\int_0^{+\infty}\int_\G\de_t\phi_t(x)+\Delta_\G\phi_t(x)\ d\mu_t\,dt+\int_\G\phi_0(x)\ d\mu_0(x)=0
\qquad
\forall\phi\in C^\infty_c([0,+\infty)\times\R^n).
\end{equation*}
By well-known results on hypoelliptic operators, this implies that $(\rho_t)_{t\ge0}$ solves the sub-elliptic heat equation $\de_t\rho_t=\Delta_\G\rho_t$ with initial datum $\rho_0$.
\end{proof}

To prove \cref{th:GF_ent_implies_heat_diff} below we need some preliminaries. The following two lemmas are natural adaptations of~\cite{AS07}*{Lemma~2.14} to our setting.

\begin{lemma}\label{lemma:AFP_func_meas}
Let $\mu\in\prob(\G)$ and $\sigma\in L^1(\G)$ with $\sigma\ge0$. Let $\nu\in\mathcal{M}(\G;\R^m)$ be a $\R^m$-valued Borel measure with finite total variation and such that $|\nu|\ll\mu$. Then
\begin{equation}\label{eq:afp_func_meas}
\int_{\G}\ \abs*{\frac{\sigma\star\nu}{\sigma\star\mu}}^2\sigma\star\mu\ dx
\le\int_{\G}\ \abs*{\frac{\nu}{\mu}}^2\, d\mu.
\end{equation}
In addition, if $(\sigma_k)_{k\in\N}\subset L^1(\G)$, $\sigma_k\ge0$, weakly converges to the Dirac mass $\delta_0$ and $\frac{\nu}{\mu}\in L^2(\G,\mu)$, then 
\begin{equation}\label{eq:afp_func_meas_limit}
\lim_{k\to+\infty}\int_{\G}\ \abs*{\frac{\sigma_k\star\nu}{\sigma_k\star\mu}}^2\sigma_k\star\mu\ dx
=\int_{\G}\ \abs*{\frac{\nu}{\mu}}^2\, d\mu.
\end{equation} 
\end{lemma}

\begin{proof}
Inequality~\eqref{eq:afp_func_meas} follows from Jensen inequality and is proved in~\cite{AS07}*{Lemma~2.14}. We briefly recall the argument for the reader's convenience. Consider the map $\Phi\colon\R^m\times\R\to[0,+\infty]$ given by
\begin{equation*}
\Phi(z,t):=
\begin{cases}
\dfrac{|z|^2}{t} & \text{if}\ t>0,\\
0 & \text{if}\ (z,t)=(0,0),\\
+\infty & \text{if either}\ t<0\ \text{or}\ t=0,\ z\ne0.
\end{cases}
\end{equation*}
Then $\Phi$ is convex, lower semicontinuous and positively 1-homogeneous. By Jensen's inequality we have
\begin{equation}\label{eq:afp_Jensen}
\Phi\left(\int_\G\psi(x)\ d\theta(x)\right)
\le\int_\G\Phi(\psi(x))\ d\theta(x)
\end{equation}
for any Borel function $\psi\colon\G\to\R^{m+1}$ and any positive and finite measure~$\theta$ on~$\G$. Fix $x\in\G$ and apply~\eqref{eq:afp_Jensen} with $\psi(y)=\left(\frac{\nu}{\mu}(y),1\right)$ and $d\theta(y)=\sigma(xy^{-1})d\mu(y)$ to obtain
\begin{align*}
\abs*{\frac{(\sigma\star\nu)(x)}{(\sigma\star\mu)(x)}}^2(\sigma\star\mu)(x) 
&=\Phi\left(\int_\G\frac{\nu}{\mu}(y)\,\sigma(xy^{-1})\ d\mu(y),\int_\G\sigma(xy^{-1})\ d\mu(y)\right)\\
&\le\int_\G\Phi\left(\frac{\nu}{\mu}(y),1\right)\sigma(xy^{-1})\ d\mu(y)
=\int_\G\abs*{\frac{\nu}{\mu}}^2(y)\,\sigma(xy^{-1})\ d\mu(y),
\end{align*}
which immediately gives~\eqref{eq:afp_func_meas}. The limit in~\eqref{eq:afp_func_meas_limit} follows by the joint lower semicontinuity of the functional $(\nu,\mu)\mapsto\int_{\G}\ \abs*{\frac{\nu}{\mu}}^2\, d\mu$, see Theorem~2.34 and Example~2.36 in~\cite{AFP00}. 
\end{proof}

In \cref{lemma:AFP_func_meas_time} below and in the rest of the paper, we let $f*g$ be the convolution of the two functions $f,g$ with respect to the time variable. We keep the notation $f\star g$ for the convolution of $f,g$ with respect to the space variable.

\begin{lemma}\label{lemma:AFP_func_meas_time}
Let $\mu_t=\rho_t\leb^n\in\prob(\G)$ for all $t\in\R$ and let $\theta\in L^1(\R)$, $\theta\ge0$. If the horizontal time-dependent vector field $v\colon\R\times\G\to H\G$ satisfies $v_t\in L^2_\G(\mu_t)$ for a.e.\ $t\in\R$, then
\begin{equation}\label{eq:afp_func_meas_time}
\int_\G\left\|\frac{\theta*(\rho_\cdot v_\cdot)(t)}{\theta*\rho_\cdot(t)}\right\|_\G^2\,\theta*\rho_\cdot(t)\ dx
\le\theta*\left(\int_\G\|v_\cdot\|_\G^2\ d\mu_\cdot\right)(t)
\qquad
\text{for all}\ t\in\R.
\end{equation}
In addition, if $(\theta_j)_{j\in\N}\subset L^1(\G)$, $\theta_j\ge0$, weakly converges to the Dirac mass $\delta_0$, then 
\begin{equation}\label{eq:afp_func_meas_limit_time}
\lim_{j\to+\infty}\int_\G\left\|\frac{\theta_j*(\rho_\cdot v_\cdot)(t)}{\theta_j*\rho_\cdot(t)}\right\|_\G^2\,\theta_j*\rho_\cdot(t)\ dx
=\int_\G\|v_t\|_\G^2\ d\mu_t
\qquad
\text{for a.e.}\ t\in\R.
\end{equation}
\end{lemma}

\begin{proof}
Inequality~\eqref{eq:afp_func_meas_time} follows from~\eqref{eq:afp_Jensen} in the same way of~\eqref{eq:afp_func_meas}, so we omit the details. For~\eqref{eq:afp_func_meas_limit_time}, set $\mu^j_t=\theta_j*\mu_\cdot(t)$ and $\nu^j_t=\theta_j*(v_\cdot\mu_\cdot)(t)$ for all $t\in\R$ and $j\in\N$. Then $\|\nu^j_t\|_\G\ll\mu^j_t$ and $\nu^j_t\weakto\nu_t=v_t\mu_t$ for a.e.\ $t\in\R$, so that
\begin{equation*}
\liminf_{j\to+\infty}\int_\G\left\|\frac{\theta_j*(\rho_\cdot v_\cdot)(t)}{\theta_j*\rho_\cdot(t)}\right\|_\G^2\,\theta_j*\rho_\cdot(t)\ dx
=\liminf_{j\to+\infty}\int_\G\left\|\frac{\nu^j_t}{\mu^j_t}\right\|_\G^2 d\mu^j_t
\ge\int_\G\left\|\frac{\nu_t}{\mu_t}\right\|_\G^2 d\mu_t
\end{equation*}
for a.e.\ $t\in\R$ by Theorem~2.34 and Example~2.36 in~\cite{AFP00}.
\end{proof}

The following lemma is an elementary result relating weak convergence and convergence of scalar products of vector fields. We prove it here for the reader's convenience. 

\begin{lemma}\label{lemma:polarization}
For $k\in\N$, let $\mu_k,\mu\in\prob(\G)$ and let $v_k,w_k,v,w\colon\G\to T\G$ be Borel vector fields. Assume that $\mu_k\weakto\mu$, $v_k\mu_k\weakto v\mu$ and $w_k\mu_k\weakto w\mu$ as $k\to+\infty$. If
\begin{equation*}
\limsup_{k\to+\infty}\int_\G\|v_k\|_\G^2\ d\mu_k\le\int_\G\|v\|_\G^2\ d\mu<+\infty
\qquad\text{and}\qquad
\limsup_{k\to+\infty}\int_\G\|w_k\|_\G^2\ d\mu_k<+\infty,
\end{equation*}
then
\begin{equation}\label{eq:polarization}
\lim_{k\to+\infty}\int_\G\scalar*{v_k,w_k}_\G\ d\mu_k=\int_\G\scalar*{v,w}_\G\ d\mu.
\end{equation}
\end{lemma}

\begin{proof}
By lower semicontinuity, we know that $\lim\limits_{k\to+\infty}\int_\G\|v_k\|_\G^2\ d\mu_k=\int_\G\|v\|^2_\G\ d\mu$ and
\begin{equation*}
\liminf_{k\to+\infty}\int_\G\|tv_k+w_k\|_\G^2\ d\mu_k
\geq
\int_\G\|tv+w\|_\G^2\ d\mu
\qquad
\text{for all}\ t\in\R.
\end{equation*}
Expanding the squares, we get
\begin{equation*}
\liminf_{k\to+\infty}\left(2t\int_\G\scalar*{v_k,w_k}_\G\ d\mu_k+\int_\G\|w_k\|_\G^2\ d\mu_k\right)
\geq 
2t\int_\G\scalar*{v,w}_\G\ d\mu
\qquad
\text{for all}\ t\in\R.
\end{equation*}
Choosing $t>0$, dividing both sides by~$t$ and letting $t\to+\infty$ gives the $\liminf$ inequality in~\eqref{eq:polarization}. Choosing $t<0$, a similar argument gives the $\limsup$ inequality in~\eqref{eq:polarization}.
\end{proof}

We are now ready to prove the second part of \cref{th:main}.

\begin{theorem}\label{th:GF_ent_implies_heat_diff}
Let $\rho_0\in L^1(\G)$ be such that $\mu_0=\rho_0\leb^n\in\dom(\ent)$. If $(\mu_t)_{t\ge0}$ is a gradient flow of $\ent$ in $(\prob_2(\G),\W_\G)$ starting from~$\mu_0$, then  $\mu_t=\rho_t\leb^n$ for all $t\ge0$ and $(\rho_t)_{t\ge0}$ solves the sub-elliptic heat equation $\de_t\rho_t=\Delta_\G\rho_t$ with initial datum $\rho_0$. In particular, $t\mapsto\ent(\mu_t)$ is locally absolutely continuous on $(0,+\infty)$.
\end{theorem}

\begin{proof}
By \cref{lemma:GF_ent_plus_AC_is_heat_diff}, we just need to show that the map $t\mapsto\ent(\mu_t)$ satisfies~\eqref{eq:GF_ent_plus_AC}. It is not restrictive to extend $(\mu_t)_{t\ge0}$ in time to the whole~$\R$ by setting $\mu_t=\mu_0$ for all $t\le0$. So from now on we assume $\mu_t\in AC^2_{\rm loc}(\R;(\prob_2(\G),\W_\G))$. The time-dependent vector field $(v^\G)_{t>0}$ given by \cref{prop:CE} extends to the whole~$\R$ accordingly. Note that $(\mu_t)_{t\in\R}$ is a gradient flow of~$\ent$ in the following sense: for each $h\in\R$, $(\mu_{t+h})_{t\ge0}$ is a gradient flow on~$\ent$ starting from~$\mu_h$. By \cref{def:metric_GF}, we get $t\mapsto|\D_\G^-\ent|(\mu_t)\in L^2_{\rm loc}(\R)$ , so that $t\mapsto\F_\G(\rho_t)\in L^1_{\rm loc}(\R)$ by \cref{prop:slope_ent_G_general}.

We divide the proof in three main steps.

\medskip

\textit{Step~1: smoothing in the time variable}. Let $\theta\colon\R\to\R$ be a symmetric smooth mollifier in~$\R$, i.e.\ 
\begin{equation*}
\theta\in C^\infty_c(\R),\quad
\supp\theta\subset[-1,1],\quad 
0\le\theta\le1,\quad
\int_{\R}\theta(t)\ dt=1.
\end{equation*}
We set $\theta_j(t):=j\,\theta(jt)$ for all $t\in\R$ and $j\in\N$. We define
\begin{equation*}
\mu_t^j:=\rho^j_t\leb^n,\qquad \rho^j_t:=(\theta_j*\rho_\cdot)(t)=\int_\R\theta_j(t-s)\,\rho_s\ ds
\qquad
\forall t\in\R,\ \forall j\in\N.
\end{equation*}

For any $s,t\in\R$, let $\pi_{s,t}\in\Gamma_0(\mu_s,\mu_t)$ be an optimal coupling between~$\mu_s$ and~$\mu_t$. An easy computation shows that $\pi^j_t\in\prob(\G\times\G)$ given by
\begin{equation*}
\int_{\G\times\G}\phi(x,y)\ d\pi^j_t(x,y)=\int_\R\theta_j(t-s)\int_{\G\times\G}\phi(x,y)\ d\pi_{s,t}(x,y)\,ds,
\end{equation*}
for any $\phi\colon\G\times\G\to[0,+\infty)$ Borel, is a coupling between~$\mu^j_t$ and~$\mu_t$. Hence we get
\begin{equation*}
\W_\G(\mu^j_t,\mu_t)^2\le\int_\R\theta_j(t-s)\,\W_\G(\mu_s,\mu_t)^2\ ds
\qquad
\forall t\in\R,\ \forall j\in\N.
\end{equation*}
Therefore $\lim_{j\to+\infty}\W_\G(\mu_t^j,\mu_t)=0$ for all $t\in\R$. This implies that $\mu^j_t\weakto\mu_t$ as $j\to+\infty$ and
\begin{equation}\label{eq:2nd_moment_mu_j}
\lim_{j\to+\infty}\int_\G\dcc(x,0)^2\ d\mu_t^j(x)=\int_\G\dcc(x,0)^2\ d\mu_t(x)
\qquad
\forall t\in\R.
\end{equation}
In particular, $(\mu^j_t)_{t\in\R}\subset\prob_2(\G)$, $\ent(\mu^j_t)>-\infty$ for all $j\in\N$ and
\begin{equation}\label{eq:mu_j_liminf}
\liminf_{j\to+\infty}\ent(\mu^j_t)\ge\ent(\mu_t)
\qquad
\forall t\in\R.
\end{equation}
We claim that
\begin{equation}\label{eq:mu_j_limsup}
\limsup_{j\to+\infty}\ent(\mu^j_t)\le\ent(\mu_t)
\qquad
\text{for a.e.}\ t\in\R.
\end{equation}
Indeed, define the new reference measure $\nu:=e^{-c\,\dcc^2(\cdot,0)}\leb^n$, where $c>0$ is chosen so that $\nu\in\prob(\G)$. Since the function $\hat{H}(r):=r\log r+(1-r)$, for $r\ge0$, is convex and non-negative, by Jensen's inequality we have
\begin{align*}
\ent_\nu(\mu^j_t)&
=\int_\G\hat{H}\left(e^{c\,\dcc^2(\cdot,0)}\,\theta_j*\rho_\cdot(t)\right)\,d\nu
=\int_\G\hat{H}\left(\theta_j*\left(e^{c\,\dcc^2(\cdot,0)}\rho_\cdot\right)(t)\right)\,d\nu\\
&\le\int_\G\theta_j*\hat{H}\left(e^{c\,\dcc^2(\cdot,0)}\rho_\cdot\right)(t)\,d\nu
=\theta_j*\ent_\nu(\mu_\cdot)(t)
\qquad
\forall t\in\R,\ \forall j\in\N.
\end{align*}
Therefore $\limsup_{j\to+\infty}\ent_\nu(\mu^j_t)\le\ent_\nu(\mu_t)$ for a.e.\ $t\in\R$. Thus~\eqref{eq:mu_j_limsup} follows by~\eqref{eq:ent_useful_formula} and~\eqref{eq:2nd_moment_mu_j}. Combining~\eqref{eq:mu_j_liminf} and~\eqref{eq:mu_j_limsup}, we get 
\begin{equation}\label{eq:mu_j_lim}
\lim_{j\to+\infty}\ent(\mu^j_t)=\ent(\mu_t)
\qquad
\text{for a.e.}\ t\in\R.
\end{equation}

Let $(v^\G_t)_{t\in\R}$ be the horizontal time-dependent vector field relative to~$(\mu_t)_{t\in\R}$ given by \cref{prop:CE}. Let $(v^j_t)_{t\in\R}$ be the horizontal time-dependent vector field given by
\begin{equation}\label{eq:def_v_j}
v^j_t=\frac{\theta_j*(\rho_\cdot v_\cdot)(t)}{\rho^j_t}
\qquad
\forall t\in\R.
\end{equation}
We claim that $v_t^j\in L^2_\G(\mu_t^j)$ for all $t\in\R$. Indeed, applying \cref{lemma:AFP_func_meas_time}, we get
\begin{equation*}
\int_\G\|v^j_t\|_\G^2\ d\mu^j_t
\le\theta_j*\left(\int_\G\|v_\cdot\|_\G^2\ d\mu_\cdot\right)(t)
=\theta_j*|\dot{\mu}_\cdot|_\G^2(t)
\qquad
\forall t\in\R.
\end{equation*}
We also claim that $(\mu^j_t)_{t\in\R}$ solves $\de_t\mu^j_t+\diverg(v^j_t\mu^j_t)=0$ in the sense of distributions for all $j\in\N$. Indeed, if $\phi\in C^\infty_c(\R\times\R^n)$, then also $\phi^j:=\theta_j*\phi\in C^\infty_c(\R\times\R^n)$, so that
\begin{align*}
&\int_\R\int_\G\de_t\phi^j(t,x)+\scalar*{\nabla_\G\phi^j(t,x),v^j_t(x)}_\G\,d\mu^j_t(x)\,dt=\\
&\hspace*{1cm}=\int_\R\theta_j*\left(\int_\G\de_t\phi(\cdot,x)+\scalar*{\nabla_\G\phi(\cdot,x),v_\cdot(x)}_\G\ d\mu_\cdot(x)\right)(t)\,dt\\
&\hspace*{1cm}=\int_\R\int_\G\de_t\phi(t,x)+\scalar*{\nabla_\G\phi(t,x),v_t(x)}_\G\,d\mu_t(x)\,dt=0
\qquad
\forall j\in\N.
\end{align*}
By \cref{prop:CE}, we conclude that $(\mu^j_t)_t\in AC^2_{\rm loc}(\R;(\prob_2(\G),\W_\G))$ with $|\dot{\mu}^j_t|_\G^2\le\theta_j*|\dot{\mu}_\cdot|_\G^2(t)$ for all $t\in\R$ and $j\in\N$.

Finally, we claim that $\F_\G(\rho^j_t)\le\theta_j*\F_\G(\rho_\cdot)(t)$ for all $t\in\R$ and $j\in\N$. Indeed, arguing as in~\eqref{eq:C-S_for_horizontal_grad}, by Cauchy--Schwarz inequality we have
\begin{equation*}
\|\nabla_\G\rho^j_t\|_\G^2
\le\left[\theta_j*\left(\chi_{\set*{\rho_.>0}}\sqrt{\rho_\cdot}\,\frac{\|\nabla_\G\rho_\cdot\|_\G}{\sqrt{\rho_\cdot}}\right)\right]^2(t)
\le\rho^j_t\,\theta_j*\left(\frac{\|\nabla_\G\rho_\cdot\|_\G^2}{\rho_\cdot}\chi_{\set*{\rho_\cdot>0}}\right)(t),
\end{equation*}
so that
\begin{equation*}
\F_\G(\rho^j_t)\le\theta_j*\left(\int_{\set*{\rho_\cdot>0}}\frac{\|\nabla_\G\rho_\cdot\|_\G^2}{\rho_\cdot}\ dx\right)(t)=\theta_j*\F_\G(\rho_\cdot)(t).
\end{equation*}

\medskip

\textit{Step~2: smoothing in the space variable}. 
Let $j\in\N$ be fixed. Let $\eta\colon\G\to\R$ be a symmetric smooth mollifier in $\G$, i.e.\ a function $\eta\in C^\infty_c(\R^n)$ such that
\begin{equation*}
\supp\eta\subset B_1,\quad 
0\le\eta\le1,\quad
\eta(x^{-1})=\eta(x)\ \forall x\in\G,\quad
\int_{\G}\eta(x)\ dx=1.
\end{equation*}
We set $\eta_k(x):=k^Q\,\eta(\delta_k(x))$ for all $x\in\G$ and $k\in\N$. We define
\begin{equation*}
\mu_t^{j,k}:=\rho^{j,k}_t\leb^n,\qquad \rho^{j,k}_t(x):=\eta_k\star\rho^j_t(x)=\int_\G\eta_k(xy^{-1})\rho^j_t(y)\ dy, \quad x\in\G,
\end{equation*}
for all $t\in\R$ and $k\in\N$. Note that
\begin{equation}\label{eq:mu_k_left_translation}
\mu^{j,k}_t=\int_\G (l_y)_\#\mu_t^j\ \eta_k(y)dy
\qquad
\forall t\in\R,\ \forall k\in\N,
\end{equation} 
where $l_y(x)=yx$, $x,y\in\G$, denotes the left-translation.

Note that $(\mu_t^{j,k})_{t\in\R}\subset\prob_2(\G)$ for all $k\in\N$. Indeed, arguing as in~\eqref{eq:heat_second_moment}, we have
\begin{equation}\label{eq:2nd_moment_mu_k}
\begin{split}
\int_\G\dcc(x,0)^2\ d\mu^{j,k}_t(x)
&=\int_\G\left(\dcc(\cdot,0)^2\star\eta_k\right)(x)\ d\mu^j_t(x)\\
&\le 2\int_\G\dcc(x,0)^2\ d\mu^j_t(x)+2\int_\G\dcc(x,0)^2\,\eta_k(x)\,dx
\end{split}
\end{equation}
for all $t\in\R$. In particular, $\ent(\mu^{j,k}_t)>-\infty$ for all $t\in\R$ and $k\in\N$. Clearly $\mu^{j,k}_t\weakto\mu^j_t$ as $k\to+\infty$ for each fixed $t\in\R$, so that
\begin{equation}\label{eq:mu_k_liminf}
\liminf_{k\to+\infty}\ent(\mu^{j,k}_t)\ge\ent(\mu^j_t)
\qquad
\forall t\in\R.
\end{equation}
Moreover, we claim that
\begin{equation}\label{eq:mu_k_limsup}
\limsup_{k\to+\infty}\ent(\mu^{j,k}_t)\le\ent(\mu^j_t)
\qquad
\forall t\in\R.
\end{equation}
Indeed, let $\nu\in\prob(\G)$ and $\hat{H}$ as in \textit{Step~1}. Recalling~\eqref{eq:mu_k_left_translation}, by Jensen's inequality we get
\begin{equation}\label{eq:mu_k_ent_push-forward}
\ent_\nu(\mu^{j,k}_t)\le\int_\G \ent_\nu((l_y)_\#\mu_t^j)\ \eta_k(y)dy.
\end{equation}
Define $\nu_y:=(l_y)_\#\nu$ and note that $\nu_y=e^{-c\,\dcc(\cdot,y)}\leb^n$ for all $y\in\G$. Thus by~\eqref{eq:ent_push-forward_formula} we have
\begin{equation*}
\ent_\nu((l_y)_\#\mu_t^j)=\ent_{\nu_{y^{-1}}}(\mu_t^j)=\ent(\mu^j_t)+c\int_\G\dcc(yx,0)^2\ d\mu^j_t(x)
\qquad
\forall y\in\G.
\end{equation*}
By the dominated convergence theorem we get that $y\mapsto\ent_\nu((l_y)_\#\mu_t^j)$ is continuous and therefore~\eqref{eq:mu_k_limsup} follows by passing to the limit as $k\to+\infty$ in~\eqref{eq:mu_k_ent_push-forward}. Combining~\eqref{eq:mu_k_liminf} and~\eqref{eq:mu_k_limsup}, we get 
\begin{equation}\label{eq:mu_k_lim}
\lim_{k\to+\infty}\ent(\mu^{j,k}_t)=\ent(\mu^j_t)
\qquad
\forall t\in\R.
\end{equation}
 
Let $(v^j_t)_{t\in\R}$ be as in~\eqref{eq:def_v_j} and let $(v^{j,k}_t)_{t\in\R}$ be the horizontal time-dependent vector field given by
\begin{equation}\label{eq:def_v_k}
v^{j,k}_t=\frac{\eta_k\star(\rho^j_t v^j_t)}{\rho^{j,k}_t}
\qquad
\forall t\in\R,\ \forall k\in\N.
\end{equation}
We claim that $v_t^{j,k}\in L^2_\G(\mu_t^{j,k})$ for all $t\in\R$. Indeed, applying \cref{lemma:AFP_func_meas}, we get
\begin{equation}\label{eq:mu_k_velocity}
\int_\G\|v^{j,k}_t\|_\G^2\ d\mu^{j,k}_t\le\int_\G\|v^j_t\|_\G^2\ d\mu^j_t\le|\dot{\mu}^j_t|_\G^2
\qquad
\forall t\in\R,\ \forall k\in\N.
\end{equation}
We also claim that $(\mu^{j,k}_t)_{t\in\R}$ solves $\de_t\mu^{j,k}_t+\diverg(v^{j,k}_t\mu^{j,k}_t)=0$ in the sense of distributions for all $k\in\N$. Indeed, if $\phi\in C^\infty_c(\R\times\R^n)$, then also $\phi^k:=\eta_k\star\phi\in C^\infty_c(\R\times\R^n)$, so that
\begin{align*}
&\int_\R\int_\G\de_t\phi^k(t,x)+\scalar*{\nabla_\G\phi^k(t,x),v^{j,k}_t(x)}_\G\,d\mu^{j,k}_t(x)\,dt=\\
&\hspace*{1cm}=\int_\G\eta_k\star\left(\int_\R\de_t\phi(t,\cdot)+\scalar*{\nabla_\G\phi(t,\cdot),v^j_t(\cdot)}_\G\,\rho^j_t(\cdot)\,dt\right)(x)\ dx\\
&\hspace*{1cm}=\int_\R\int_\G\de_t\phi(t,x)+\scalar*{\nabla_\G\phi(t,x),v^j_t(x)}_\G\,d\mu^j_t(x)\,dt=0
\qquad
\forall k\in\N.
\end{align*}
Here we have exploited a key property of the space $(\G,\dcc,\leb^n)$ which cannot be expected in a general metric measure space, that is, the continuity equation in~\eqref{eq:CE} is preserved under regularization in the space variable. By \cref{prop:CE} and~\eqref{eq:mu_k_velocity}, we conclude that $(\mu^{j,k}_t)_t\in AC^2_{\rm loc}(\R;(\prob_2(\G),\W_\G))$ with $|\dot{\mu}^{j,k}_t|_\G\le\|v^{j,k}_t\|_{L^2_\G(\mu^{j,k}_t)}\le\|v^j_t\|_{L^2_\G(\mu^j_t)}\le|\dot{\mu}^j_t|_\G$ for all $t\in\R$ and $k\in\N$.

Finally, we claim that $\F_\G(\rho^{j,k}_t)\le\F_\G(\rho^j_t)$ for all $t\in\R$ and $k\in\N$. Indeed, arguing as in~\eqref{eq:C-S_for_horizontal_grad}, by Cauchy--Schwarz inequality we have
\begin{equation*}
\|\nabla_\G\rho^{j,k}_t\|_\G^2
\le\left[\eta_k\star\left(\chi_{\set*{\rho^j_t>0}}\sqrt{\rho^j_t}\,\frac{\|\nabla_\G\rho^j_t\|_\G}{\sqrt{\rho^j_t}}\right)\right]^2
\le\rho^{j,k}_t\,\eta_k\star\left(\frac{\|\nabla_\G\rho^j_t\|_\G^2}{\rho^j_t}\chi_{\set*{\rho^j_t>0}}\right),
\end{equation*}
so that
\begin{align*}
\F_\G(\rho^{j,k}_t)
\le\int_\G\eta_k\star\left(\frac{\|\nabla_\G\rho^j_t\|_\G^2}{\rho^j_t}\chi_{\set*{\rho^j_t>0}}\right)\ dx
=\int_{\set*{\rho^j_t>0}}\frac{\|\nabla_\G\rho^j_t\|_\G^2}{\rho^j_t}\ dx
=\F_\G(\rho^j_t).
\end{align*} 

\medskip

\textit{Step~3: truncated entropy}. 
Let $j,k\in\N$ be fixed. For any $m\in\N$, consider the maps $z_m,H_m\colon[0,+\infty)\to\R$ defined in~\eqref{eq:def_truncated_entropy}. We set $\tilde{z}_m(r)=z_m(r)+m$ for all $r\ge0$ and $m\in\N$. Since $\rho^{j,k}_t\in\prob(\G)$ for all $t\in\R$, differentiating under the integral sign we get
\begin{equation*}
\frac{d}{dt}\int_\G H_m(\rho^{j,k}_t)\ dx
=\int_\G \tilde{z}_m(\rho^{j,k}_t)\,\de_t\rho^{j,k}_t\ dx
\end{equation*}
for all $t\in\R$. Fix $t_0,t_1\in\R$ with $t_0<t_1$. Then
\begin{equation}\label{eq:truncated_ent_before_AC}
\int_\G H_m(\rho^{j,k}_{t_1})\ dx-\int_\G H_m(\rho^{j,k}_{t_0})\ dx
=\int_{t_0}^{t_1}\int_\G \tilde{z}_m(\rho^{j,k}_t)\,\de_t\rho^{j,k}_t\,dxdt.
\end{equation}

Let $(\alpha_i)_{i\in\N}\subset C^\infty_c(t_0,t_1)$ such that $0\le\alpha_i\le1$ and $\alpha_i\to\chi_{(t_0,t_1)}$ in $L^1(\R)$ as $i\to+\infty$. Let $i\in\N$ be fixed and consider the function $u_t(x)=\tilde{z}_m(\rho^{j,k}_t(x))\,\alpha_i(t)$ for all $(t,x)\in\R\times\R^n$. We claim that there exists $(\psi^h)_{h\in\N}\subset C^\infty_c(\R^{n+1})$ such that 
\begin{equation}\label{eq:approx_for_AC}
\lim_{h\to+\infty}\int_\R\int_\G\,\abs{u_t(x)-\psi^h_t(x)}^2+\|\nabla_\G u_t(x)-\nabla_\G\psi^h_t(x)\|_\G^2\ dx\,dt=0.
\end{equation}
Indeed, consider the direct product $\G^*=\R\times\G$ and note that~$\G^*$ is a Carnot group. Recalling~\eqref{eq:def_horiz_sobolev_space}, we know that $C^\infty_c(\R^{n+1})$ is dense in $W^{1,2}_{\G^*}(\R^{n+1})$. Thus, to get~\eqref{eq:approx_for_AC} we just need to prove that $u\in W^{1,2}_{\G^*}(\R^{n+1})$ (in fact, the $L^2$-integrability of $\de_t u$ is not strictly necessary in order to achieve~\eqref{eq:approx_for_AC}). We have $\rho^{j,k},\de_t\rho^{j,k}\in L^\infty(\R^{n+1})$, because
\begin{equation*}
\|\rho^{j,k}\|_{L^\infty(\R^{n+1})}\le\|\eta_k\|_{L^\infty(\R^n)},
\qquad
\|\de_t\rho^{j,k}\|_{L^\infty(\R^{n+1})}\le\|\theta_j'\|_{L^1(\R)}\|\eta_k\|_{L^\infty(\R^n)}
\end{equation*}
by Young's inequality. Moreover, $\rho^{j,k}\alpha_i,\de_t\rho^{j,k}\alpha_i\in L^1(\R^{n+1})$, because
\begin{equation*}
\|\rho^{j,k}\alpha_i\|_{L^1(\R^{n+1})}=\|\alpha_i\|_{L^1(\R)},
\qquad
\|\de_t\rho^{j,k}\alpha_i\|_{L^1(\R^{n+1})}\le\|\theta_j'\|_{L^1(\R)}\|\alpha_i\|_{L^1(\R)}.
\end{equation*} 
Therefore $\rho^{j,k}\alpha_i,\de_t\rho^{j,k}\alpha_i\in L^2(\R^{n+1})$, which immediately gives $u\in L^2(\R^{n+1})$. Now by \textit{Step~2} we have that
\begin{equation*}
\int_\G\|\nabla_\G\rho^{j,k}_t\|_\G^2\ dx
\le\|\rho^{j,k}_t\|_{L^\infty(\R^n)}\F_\G(\rho^{j,k}_t)
\le\|\rho^{j,k}_t\|_{L^\infty(\R^n)}\F_\G(\rho^j_t)
\qquad
\forall t\in\R,
\end{equation*}
so that by \textit{Step~1} we get
\begin{equation*}
\|\nabla_\G\rho^{j,k}\alpha_i\|_{L^2(\R^{n+1})}^2
\le\|\eta_k\|_{L^\infty(\R^n)}
\|\alpha_i^2\cdot\theta_j*\F_\G(\rho_\cdot)\|_{L^1(\R)}.
\end{equation*}
This prove that $\|\nabla_\G u\|_\G\in L^2(\R^{n+1})$. The previous estimates easily imply that also $\de_t u\in L^2(\R^{n+1})$. This concludes the proof of~\eqref{eq:approx_for_AC}.

Since $\rho^{j,k},\de_t\rho^{j,k}\in L^\infty(\R^{n+1})$, by~\eqref{eq:approx_for_AC} we get that
\begin{equation}\label{eq:limit_trunc_ent_1}
\lim_{h\to+\infty}\int_\R\int_\G\psi^h_t\,\de_t\rho^{j,k}_t\,dx dt
=\int_\R\alpha_i(t)\int_\G \tilde{z}_m(\rho^{j,k}_t)\,\de_t\rho^{j,k}_t\, dxdt
\end{equation}
and 
\begin{equation}\label{eq:limit_trunc_ent_2}
\lim_{h\to+\infty}\int_\R\int_\G\scalar*{\nabla_\G\psi^h_t,v^{j,k}_t}_\G\ d\mu^{j,k}_t\,dt
=\int_\R\alpha_i(t)\int_\G \tilde{z}_m'(\rho^{j,k}_t)\,\scalar*{\nabla_\G\rho^{j,k}_t,v^{j,k}_t}_\G\,d\mu^{j,k}_t\,dt
\end{equation}
for each fixed $i\in\N$. Since $\de_t\mu^{j,k}_t+\diverg(v^{j,k}_t\mu^{j,k}_t)=0$ in the sense of distributions by \textit{Step~2}, for each $h\in\N$ we have 
\begin{equation*}
\int_\R\int_\G\psi^h_t\,\de_t\rho^{j,k}_t\,dx dt
=-\int_\R\int_\G\de_t\psi^h_t\ d\mu^{j,k}_t\,dt
=\int_\R\int_\G\scalar*{\nabla_\G\psi^h_t,v^{j,k}_t}_\G\,d\mu^{j,k}_t\,dt.
\end{equation*}
We can thus combine~\eqref{eq:limit_trunc_ent_1} and~\eqref{eq:limit_trunc_ent_2} to get
\begin{equation}\label{eq:trunc_ent_i}
\int_\R\alpha_i(t)\int_\G \tilde{z}_m(\rho^{j,k}_t)\,\de_t\rho^{j,k}_t\, dxdt
=\int_\R\alpha_i(t)\int_\G \tilde{z}_m'(\rho^{j,k}_t)\,\scalar*{\nabla_\G\rho^{j,k}_t,v^{j,k}_t}_\G\,d\mu^{j,k}_t\,dt.
\end{equation}
Passing to the limit as $i\to+\infty$ in~\eqref{eq:trunc_ent_i}, we finally get that
\begin{equation}\label{eq:limit_through_AC}
\int_{t_0}^{t_1}\int_\G \tilde{z}_m(\rho^{j,k}_t)\,\de_t\rho^{j,k}_t\, dxdt
=\int_{t_0}^{t_1}\int_\G \tilde{z}_m'(\rho^{j,k}_t)\,\scalar*{\nabla_\G\rho^{j,k}_t,v^{j,k}_t}_\G\,d\mu^{j,k}_t\,dt 
\end{equation}
by the dominated convergence theorem. Combining~\eqref{eq:truncated_ent_before_AC} and~\eqref{eq:limit_through_AC}, we get
\begin{equation}\label{eq:truncated_ent_after_AC}
\int_\G H_m(\rho^{j,k}_{t_1})\ dx-\int_\G H_m(\rho^{j,k}_{t_0})\ dx=\int_{t_0}^{t_1}\int_{\set*{e^{-m-1}<\rho^{j,k}_t<e^{m-1}}}\scalar*{-w^{j,k}_t,v^{j,k}_t}_\G\,d\mu^{j,k}_t\,dt
\end{equation}
for all $t_0,t_1\in\R$ with $t_0<t_1$, where $w^{j,k}_t=\nabla_\G\rho^{j,k}_t/\rho^{j,k}_t$ in $L^2_\G(\mu^{j,k}_t)$ for all $t\in\R$. 

\bigskip

We can now conclude the proof. We pass to the limit as $m\to+\infty$ in~\eqref{eq:truncated_ent_before_AC} and we get 
\begin{equation}\label{eq:truncated_ent_limit_m}
\ent(\mu^{j,k}_{t_1})-\ent(\mu^{j,k}_{t_0})=\int_{t_0}^{t_1}\int_\G \scalar*{-w^{j,k}_t,v^{j,k}_t}_\G\,d\mu^{j,k}_t\,dt
\end{equation}
for all $t_0,t_1\in\R$ with $t_0<t_1$. For the left-hand side of~\eqref{eq:truncated_ent_before_AC}, recall~\eqref{eq:truncated_ent_prop_1} and~\eqref{eq:truncated_ent_prop_2} and apply the monotone convergence theorem on $\set*{\rho^{j,k}_t\le1}$ and the dominated convergence theorem on $\set*{\rho^{j,k}_t>1}$. For the right-hand side of~\eqref{eq:truncated_ent_before_AC}, recall that $t\mapsto\|v^{j,k}_t\|_{L^2_\G(\mu^{j,k}_t)}\in L^2_{\rm loc}(\R)$ and that $t\mapsto\|w^{j,k}_t\|_{L^2_\G(\mu^{j,k}_t)}=\F_\G^{1/2}(\rho^{j,k}_t)\in L^2_{\rm loc}(\R)$ by \textit{Step~2} and apply Cauchy--Schwarz inequality and the dominated convergence theorem.

We pass to the limit as $k\to+\infty$ in~\eqref{eq:truncated_ent_limit_m} and we get
\begin{equation}\label{eq:truncated_ent_limit_k}
\ent(\mu^j_{t_1})-\ent(\mu^j_{t_0})=\int_{t_0}^{t_1}\int_\G \scalar*{-w^j_t,v^j_t}_\G\ d\mu^j_t\,dt
\end{equation}
for all $t_0,t_1\in\R$ with $t_0<t_1$, where $w^j_t=\nabla_\G\rho^j_t/\rho^j_t$ in $L^2_\G(\mu^j_t)$ for all $t\in\R$. For the left-hand side of~\eqref{eq:truncated_ent_limit_m}, recall~\eqref{eq:mu_k_lim}. For the right-hand side of~\eqref{eq:truncated_ent_limit_m}, recall that $t\mapsto\|v^j_t\|_{L^2_\G(\mu^j_t)}\in L^2_{\rm loc}(\R)$ and that $t\mapsto\|w^j_t\|_{L^2_\G(\mu^{j,k}_t)}=\F_\G^{1/2}(\rho^j_t)\in L^2_{\rm loc}(\R)$ by \textit{Step~1}, so that the conclusion follows applying \cref{lemma:AFP_func_meas}, \cref{lemma:polarization}, Cauchy--Schwarz inequality and the dominated convergence theorem.

We finally pass to the limit as $j\to+\infty$ in~\eqref{eq:truncated_ent_limit_k} and we get
\begin{equation}\label{eq:truncated_ent_limit}
\ent(\mu_{t_1})-\ent(\mu_{t_0})=\int_{t_0}^{t_1}\int_\G \scalar*{-w^\G_t,v^\G_t}_\G\,d\mu_t\,dt
\end{equation}
for all $t_0,t_1\in\R\setminus\mathcal{N}$ with $t_0<t_1$, where $\mathcal{N}\subset\R$ is the set of discontinuity points of $t\mapsto\ent(\mu_t)$ and $w^\G_t=\nabla_\G\rho_t/\rho_t$ in $L^2_\G(\mu_t)$ for a.e.\ $t\in\R$ by \cref{prop:slope_ent_G_general}. For the left-hand side of~\eqref{eq:truncated_ent_limit_k}, recall~\eqref{eq:mu_j_lim}. For the right-hand side of~\eqref{eq:truncated_ent_limit_k}, recall that $t\mapsto\|v^\G_t\|_{L^2_\G(\mu_t)}\in L^2_{\rm loc}(\R)$ and that $t\mapsto\|w^\G_t\|_{L^2_\G(\mu^{j,k}_t)}=\F_\G^{1/2}(\rho_t)\in L^2_{\rm loc}(\R)$, so that the conclusion follows applying \cref{lemma:AFP_func_meas_time}, \cref{lemma:polarization}, Cauchy--Schwarz inequality and the dominated convergence theorem. From~\eqref{eq:truncated_ent_limit} we immediately deduce~\eqref{eq:GF_ent_plus_AC} and we can conclude the proof by \cref{lemma:GF_ent_plus_AC_is_heat_diff}. In particular, by \cref{prop:entropy_dissipation} the map $t\mapsto\ent(\mu_t)$ is locally absolutely continuous on~$(0,+\infty)$.
\end{proof}



\begin{bibdiv}
\begin{biblist}

\bib{AFP00}{book}{
   author={Ambrosio, Luigi},
   author={Fusco, Nicola},
   author={Pallara, Diego},
   title={Functions of bounded variation and free discontinuity problems},
   series={Oxford Mathematical Monographs},
   publisher={The Clarendon Press, Oxford University Press, New York},
   date={2000},
}

\bib{AG13}{article}{
   author={Ambrosio, Luigi},
   author={Gigli, Nicola},
   title={A user's guide to optimal transport},
   conference={
      title={Modelling and optimisation of flows on networks},
   },
   book={
      series={Lecture Notes in Math.},
      volume={2062},
      publisher={Springer, Heidelberg},
   },
   date={2013},
   pages={1--155},
}

\bib{AGMR15}{article}{
   author={Ambrosio, Luigi},
   author={Gigli, Nicola},
   author={Mondino, Andrea},
   author={Rajala, Tapio},
   title={Riemannian Ricci curvature lower bounds in metric measure spaces with $\sigma$-finite measure},
   journal={Trans. Amer. Math. Soc.},
   volume={367},
   date={2015},
   number={7},
   pages={4661--4701},
}

\bib{AGS08}{book}{
   author={Ambrosio, Luigi},
   author={Gigli, Nicola},
   author={Savar\'e, Giuseppe},
   title={Gradient flows in metric spaces and in the space of probability
   measures},
   series={Lectures in Mathematics ETH Z\"urich},
   edition={2},
   publisher={Birkh\"auser Verlag, Basel},
   date={2008},
}

\bib{AGS14}{article}{
   author={Ambrosio, Luigi},
   author={Gigli, Nicola},
   author={Savaré, Giuseppe},
   title={Calculus and heat flow in metric measure spaces and applications to spaces with Ricci bounds from below},
   journal={Invent. Math.},
   volume={195},
   date={2014},
   number={2},
   pages={289--391},
}

\bib{AGS14-2}{article}{
   author={Ambrosio, Luigi},
   author={Gigli, Nicola},
   author={Savar\'e, Giuseppe},
   title={Metric measure spaces with Riemannian Ricci curvature bounded from
   below},
   journal={Duke Math. J.},
   volume={163},
   date={2014},
   number={7},
   pages={1405--1490},
}

\bib{AS07}{article}{
   author={Ambrosio, Luigi},
   author={Savaré, Giuseppe},
   title={Gradient flows of probability measures},
   conference={
      title={Handbook of differential equations: evolutionary equations.
      Vol. III},
   },
   book={
      series={Handb. Differ. Equ.},
      publisher={Elsevier/North-Holland, Amsterdam},
   },
   date={2007},
   pages={1--136},
}

\bib{BB16}{article}{
   author={Baudoin, Fabrice},
   author={Bonnefont, Michel},
   title={Reverse Poincar\'e inequalities, isoperimetry, and Riesz transforms
   in Carnot groups},
   journal={Nonlinear Anal.},
   volume={131},
   date={2016},
   pages={48--59},
}

\bib{B08}{article}{
   author={Bernard, Patrick},
   title={Young measures, superposition and transport},
   journal={Indiana Univ. Math. J.},
   volume={57},
   date={2008},
   number={1},
   pages={247--275},
   issn={0022-2518},
}

\bib{BLU07}{book}{
   author={Bonfiglioli, A.},
   author={Lanconelli, E.},
   author={Uguzzoni, F.},
   title={Stratified Lie groups and potential theory for their sub-Laplacians},
   series={Springer Monographs in Mathematics},
   publisher={Springer, Berlin},
   date={2007},
}

\bib{B73}{book}{
   author={Br\'ezis, H.},
   title={Op\'erateurs maximaux monotones et semi-groupes de contractions dans
   les espaces de Hilbert},
   language={French},
   note={North-Holland Mathematics Studies, No. 5. Notas de Matem\'atica
   (50)},
   publisher={North-Holland Publishing Co., Amsterdam-London; American
   Elsevier Publishing Co., Inc., New York},
   date={1973},
}

\bib{CDPT07}{book}{
   author={Capogna, Luca},
   author={Danielli, Donatella},
   author={Pauls, Scott D.},
   author={Tyson, Jeremy T.},
   title={An introduction to the Heisenberg group and the sub-Riemannian isoperimetric problem},
   series={Progress in Mathematics},
   volume={259},
   publisher={Birkh\"auser Verlag, Basel},
   date={2007},
   pages={xvi+223},
}

\bib{CM14}{article}{
   author={Carlen, Eric A.},
   author={Maas, Jan},
   title={An analog of the 2-Wasserstein metric in non-commutative
   probability under which the fermionic Fokker-Planck equation is gradient
   flow for the entropy},
   journal={Comm. Math. Phys.},
   volume={331},
   date={2014},
   number={3},
   pages={887--926},
}

\bib{E10}{article}{
   author={Erbar, Matthias},
   title={The heat equation on manifolds as a gradient flow in the
   Wasserstein space},
   journal={Ann. Inst. Henri Poincar\'e Probab. Stat.},
   volume={46},
   date={2010},
   number={1},
   pages={1--23},
}

\bib{EM14}{article}{
   author={Erbar, Matthias},
   author={Maas, Jan},
   title={Gradient flow structures for discrete porous medium equations},
   journal={Discrete Contin. Dyn. Syst.},
   volume={34},
   date={2014},
   number={4},
   pages={1355--1374},
}

\bib{FSS10}{article}{
   author={Fang, Shizan},
   author={Shao, Jinghai},
   author={Sturm, Karl-Theodor},
   title={Wasserstein space over the Wiener space},
   journal={Probab. Theory Related Fields},
   volume={146},
   date={2010},
   number={3-4},
   pages={535--565},
}

\bib{FSSC96}{article}{
   author={Franchi, Bruno},
   author={Serapioni, Raul},
   author={Serra Cassano, Francesco},
   title={Meyers-Serrin type theorems and relaxation of variational integrals depending on vector fields},
   journal={Houston J. Math.},
   volume={22},
   date={1996},
   number={4},
   pages={859--890},
}

\bib{FSSC03}{article}{
   author={Franchi, Bruno},
   author={Serapioni, Raul},
   author={Serra Cassano, Francesco},
   title={On the structure of finite perimeter sets in step 2 Carnot groups},
   journal={J. Geom. Anal.},
   volume={13},
   date={2003},
   number={3},
   pages={421--466},
}

\bib{G10}{article}{
   author={Gigli, Nicola},
   title={On the heat flow on metric measure spaces: existence, uniqueness
   and stability},
   journal={Calc. Var. Partial Differential Equations},
   volume={39},
   date={2010},
   number={1-2},
   pages={101--120},
}

\bib{GH15}{article}{
   author={Gigli, Nicola},
   author={Han, Bang-Xian},
   title={The continuity equation on metric measure spaces},
   journal={Calc. Var. Partial Differential Equations},
   volume={53},
   date={2015},
   number={1-2},
   pages={149--177},
}

\bib{GKO13}{article}{
   author={Gigli, Nicola},
   author={Kuwada, Kazumasa},
   author={Ohta, Shin-Ichi},
   title={Heat flow on Alexandrov spaces},
   journal={Comm. Pure Appl. Math.},
   volume={66},
   date={2013},
   number={3},
   pages={307--331},
}

\bib{GO12}{article}{
   author={Gigli, Nicola},
   author={Ohta, Shin-Ichi},
   title={First variation formula in Wasserstein spaces over compact
   Alexandrov spaces},
   journal={Canad. Math. Bull.},
   volume={55},
   date={2012},
   number={4},
   pages={723--735},
}

\bib{JKO98}{article}{
   author={Jordan, Richard},
   author={Kinderlehrer, David},
   author={Otto, Felix},
   title={The variational formulation of the Fokker-Planck equation},
   journal={SIAM J. Math. Anal.},
   volume={29},
   date={1998},
   number={1},
   pages={1--17},
}

\bib{J09}{article}{
   author={Juillet, Nicolas},
   title={Geometric inequalities and generalized Ricci bounds in the
   Heisenberg group},
   journal={Int. Math. Res. Not. IMRN},
   date={2009},
   number={13},
   pages={2347--2373},
}


\bib{J14}{article}{
   author={Juillet, Nicolas},
   title={Diffusion by optimal transport in Heisenberg groups},
   journal={Calc. Var. Partial Differential Equations},
   volume={50},
   date={2014},
   number={3-4},
   pages={693--721},
}

\bib{KL09}{article}{
   author={Khesin, Boris},
   author={Lee, Paul},
   title={A nonholonomic Moser theorem and optimal transport},
   journal={J. Symplectic Geom.},
   volume={7},
   date={2009},
   number={4},
   pages={381--414},
}

\bib{LeD16}{article}{
   author={Le Donne, Enrico},
   title={A primer on Carnot groups: homogenous groups, Carnot-Carath\'eodory
   spaces, and regularity of their isometries},
   journal={Anal. Geom. Metr. Spaces},
   volume={5},
   date={2017},
   pages={116--137},
}

\bib{L06}{article}{
   author={Li, Hong-Quan},
   title={Estimation optimale du gradient du semi-groupe de la chaleur sur le groupe de Heisenberg},
   journal={J. Funct. Anal.},
   volume={236},
   date={2006},
   number={2},
   pages={369--394},
}

\bib{L07}{article}{
   author={Li, Hong-Quan},
   title={Estimations asymptotiques du noyau de la chaleur sur les groupes de Heisenberg},
   journal={C. R. Math. Acad. Sci. Paris},
   volume={344},
   date={2007},
   number={8},
   pages={497--502},
}

\bib{M11}{article}{
   author={Maas, Jan},
   title={Gradient flows of the entropy for finite Markov chains},
   journal={J. Funct. Anal.},
   volume={261},
   date={2011},
   number={8},
   pages={2250--2292},
}

\bib{M76}{article}{
   author={Milnor, John},
   title={Curvatures of left invariant metrics on Lie groups},
   journal={Advances in Math.},
   volume={21},
   date={1976},
   number={3},
   pages={293--329},
}

\bib{M02}{book}{
   author={Montgomery, Richard},
   title={A tour of subriemannian geometries, their geodesics and
   applications},
   series={Mathematical Surveys and Monographs},
   volume={91},
   publisher={American Mathematical Society, Providence, RI},
   date={2002},
   pages={xx+259},
}

\bib{O09}{article}{
   author={Ohta, Shin-ichi},
   title={Gradient flows on Wasserstein spaces over compact Alexandrov
   spaces},
   journal={Amer. J. Math.},
   volume={131},
   date={2009},
   number={2},
   pages={475--516},
}


\bib{OS09}{article}{
   author={Ohta, Shin-ichi},
   author={Sturm, Karl-Theodor},
   title={Heat flow on Finsler manifolds},
   journal={Comm. Pure Appl. Math.},
   volume={62},
   date={2009},
   number={10},
   pages={1386--1433},
}

\bib{O01}{article}{
   author={Otto, Felix},
   title={The geometry of dissipative evolution equations: the porous medium
   equation},
   journal={Comm. Partial Differential Equations},
   volume={26},
   date={2001},
   number={1-2},
   pages={101--174},
}

\bib{P11}{article}{
   author={Petrunin, Anton},
   title={Alexandrov meets Lott-Villani-Sturm},
   journal={M\"unster J. Math.},
   volume={4},
   date={2011},
   pages={53--64},
}

\bib{vRS05}{article}{
   author={von Renesse, Max-K.},
   author={Sturm, Karl-Theodor},
   title={Transport inequalities, gradient estimates, entropy, and Ricci curvature},
   journal={Comm. Pure Appl. Math.},
   volume={58},
   date={2005},
   number={7},
   pages={923--940},
}

\bib{S07}{article}{
   author={Savar\'e, Giuseppe},
   title={Gradient flows and diffusion semigroups in metric spaces under
   lower curvature bounds},
   journal={C. R. Math. Acad. Sci. Paris},
   volume={345},
   date={2007},
   number={3},
   pages={151--154},
}

\bib{VSC92}{book}{
   author={Varopoulos, N. Th.},
   author={Saloff-Coste, L.},
   author={Coulhon, T.},
   title={Analysis and geometry on groups},
   series={Cambridge Tracts in Mathematics},
   volume={100},
   publisher={Cambridge University Press, Cambridge},
   date={1992},
}

\bib{V09}{book}{
   author={Villani, Cédric},
   title={Optimal transport, old and new},
   series={Fundamental Principles of Mathematical Sciences},
   volume={338},
   publisher={Springer-Verlag, Berlin},
   date={2009},
}

\bib{W11}{article}{
   author={Wang, Feng-Yu},
   title={Equivalent semigroup properties for the curvature-dimension
   condition},
   journal={Bull. Sci. Math.},
   volume={135},
   date={2011},
   number={6-7},
   pages={803--815},
}

\end{biblist}
\end{bibdiv}
\end{document}